\documentclass[11pt,a4paper]{article}
\usepackage[inactive, tightpage]{preview}
\usepackage{jheppub_modified}
\usepackage{amsthm}
\usepackage[T1]{fontenc}
\usepackage[utf8]{inputenc}
\usepackage{mathtools}   
\usepackage{amsfonts}
\usepackage{physics}
\usepackage{mathrsfs}
\usepackage{natbib}
\usepackage{tikz, pgfplots} 
\usetikzlibrary{patterns, calc, patterns, angles, quotes, decorations, shadows}
\usepackage{enumitem}
\usepackage{subcaption}
\usepackage[font={small}]{caption}
\usepackage{graphicx}

\usepackage{microtype}
\usepackage{hyperref}
\definecolor{darkmidnightblue}{rgb}{0.0, 0.2, 0.4}
\hypersetup{
    colorlinks = true, 
    linkcolor=darkmidnightblue,
    citecolor=darkmidnightblue
}

\usepackage{cleveref}

\usepackage{booktabs}
\usepackage{pdflscape}
\usepackage{afterpage}
\usepackage{flafter}
\usepackage{rotating}
\usepackage{fancyhdr}
\fancypagestyle{lscape}{%
\fancyhf{} 
\fancyfoot[LE]{}
\fancyfoot[LO] {}
 
}

\newcommand{\B}{\mathcal{B}}
\newcommand{\ep}{\epsilon}
\newcommand{\e}{\mathrm{e}}
\renewcommand{\i}{\mathrm{i}}
\newcommand{\im}{\mathrm{i}}
\newcommand*{\de}{\operatorname{d\!}{}} 
\renewcommand{\dd}[2]{\frac{\de#1}{\de#2}}

\title{Resurgent aspects of applied exponential asymptotics}

\author{Samuel Crew \& Philippe H. Trinh}
\affiliation{Department of Mathematical Sciences, University of Bath\\
Claverton Down, Bath BA2 7AY, UK}

\emailAdd{sc754@bath.ac.uk}
\emailAdd{p.trinh@bath.ac.uk}

\abstract{
In many physical problems, it is important to capture exponentially-small effects that lie beyond-all-orders of a typical asymptotic expansion; when collected, the full expansion is known as the trans-series. Applied exponential asymptotics has been enormously successful in developing  practical tools for studying the leading exponentials of a trans-series expansion, typically in the context of singular non-linear perturbative differential or integral equations. 
Separate to applied exponential asymptotics, there exists a closely related line of development known as \'Ecalle's theory of resurgence, which describes the connection between trans-series and a certain class of holomorphic functions known as resurgent functions. This connection is realised through the process of Borel resummation. 
However, in contrast to singularly perturbed problems, Borel resummation and \'Ecalle's resurgence theory have mainly focused on non-parametric asymptotic expansions (i.e. differential equations without a parameter). The relationships between these latter areas and applied exponential asymptotics has not been thoroughly examined, partially due to differences in language and emphasis.
In this work, we explore these connections by developing an alternative framework for the factorial-over-power ansatz in exponential asymptotics that is centred on the Borel plane. Our work clarifies a number of elements used in apploed exponential asymptotics, such as the heuristic use of Van Dyke's rule and the universality of factorial-over-power ansatzes. Along the way, we provide a number of useful tools for probing more pathological problems in exponential asymptotics known to arise in applications; this includes problems with coalescing singularities, nested boundary layers, and more general late-term behaviours. 
}

\theoremstyle{definition}

\newtheorem{corollary}{Corollary}[section]
\newtheorem{lemma}{Lemma}[section]

\newtheorem{example}{Example}[section]
\newtheorem*{remark}{Remark}

\begin{document}
\maketitle
\newpage

\section{Introduction}
Exact solutions to interacting physical systems are rare. This sparsity persists in nature across a wide range of energy scales, from strongly interacting gauge and string theories, through to quantum mechanics and classical effective field theories such as the Navier-Stokes equations. Unless a model enjoys some additional symmetry, then typically closed form expressions for generic physical observables (be that correlation functions in quantum field theory or free surface profiles in fluid flow) cannot be found. Instead, we often turn to studying observables at isolated points in some parameter space where non-dimensional parameters are small (or large). Typically, we suppose a model will be exactly soluble at some trivial point $\epsilon=0$, and we look to extrapolate observables, say $y$, across the whole space by seeking an expansion in a small parameter,
\begin{equation} \label{eq:observable}
    y \sim y_0 + y_1 \epsilon + y_2 \epsilon^2 \ldots
\end{equation}
Indeed, one may hope that such expansions may be patched together to cover the entire parameter space. 

The method of seeking an expansion of physical quantities in small parameters has proved very powerful in the physical sciences. Unfortunately, one eventually runs into a problem with such a perturbative approach. In a wide variety of applications, the series \eqref{eq:observable} will have a zero radius of convergence. While such series can be tremendously accurate with few terms \cite{bender2013advanced}, eventually the series will diverge, and one is faced with the apparent impossibility of recovering an analytic extrapolation in the sense hoped-for above. The physical mechanisms for such divergence are diverse. For example, in quantum field theory, a heuristic is the factorial growth of Feynman diagrams \cite{le2012large} or Dyson's \cite{dyson1952divergence} instability argument for the zero radius of convergence for series \eqref{eq:observable} arising in quantum electrodynamics. In applied mathematics and classical physics, expansions such as \eqref{eq:observable} may arise as singular perturbations, where later terms in $y$ depend on derivatives of earlier terms, and factorial divergence is born in a more prescriptive way (see \emph{e.g.}~\cite{chapman_1998_exponential_asymptotics}).

However, all is not lost. It was soon realised that perturbative series contain much more information than first appears. In particular, an argument, principally due to Berry~\cite{berry_1989} and Dingle \cite{dingle1973asymptotic}, shows that when one truncates \eqref{eq:observable} at the optimal point before it begins to diverge, then there are hidden non-perturbative contributions to the series of the form $e^{-1/\epsilon}$. As $\epsilon$ varies in the complex plane, terms such as these are smoothly switched-on across so-called \emph{Stokes lines}; although they begin exponentially small, eventually they may come to dominate the asymptotics in other sectors of $\mathbb{C}$. It is a breathtaking result that hiding in the late (divergent) perturbative coefficients is the non-perturbative physics. Resummation methods allow one to understand such non-analytic additions by assigning an analytic value to $y$. Borel resummation, in particular, realises $y$ as an asymptotic expansion of a contour integral of an auxiliary function, the \textit{Borel transform}, in \textit{Borel space}. The aforementioned ambiguities related to divergence are then associated with the treatment of singularities of the Borel transform. Resummation methods have enjoyed enormous success in the analysis of differential equations with singular points (see \emph{e.g.}~\cite{costin2008asymptotics}). Furthermore, by now, there is a large literature (see \cite{Aniceto:2018bis} for a review) on Borel resummation in theoretical physics; in many such cases, the Borel poles are physical, and correspond to non-perturbative semi-classical contributions to the path integral (e.g. instantons). The potential of using such ideas in the development of non-perturbative quantum field theory cannot be understated.

In the seminal work of \'Ecalle \cite{ecalle1981fonctions} it was shown that trans-series and resummation techniques can be phrased in the language of complex analysis. \'Ecalle elucidates a beautiful correspondence between trans-series and holomorphic (resurgent) functions, i.e.
\begin{equation}\label{eq:correspondence}
    \text{Trans-series,} \sum_i e^{-\chi_i/\epsilon}y^{(i)}(\epsilon) \longleftrightarrow \text{Holomorphic (resurgent) functions, }y_B(z).
\end{equation}
In particular, under this correspondence, the notion that perturbative series know about their own non-perturbative completion is merely analytic continuation by another name. The theory of hyperterminants \cite{berry1991hyperasymptotics,berry1990hyperasymptotics,olde1996hyperterminants,olde1998hyperterminants} mirrors the more formal resurgence theory for beyond-all-orders saddle-point evaluation of integrals. 

\subsection{Goals of this work}
We have three main goals for this work. First, we shall apply some aspects of resurgence to the study of singularly perturbed linear ordinary differential equations. An example is 
\begin{equation} \label{eq:intro_ode}
    \epsilon^2 y''(z) + \epsilon y'(z) + y(z)= 1/z \,,
\end{equation}
with $z\in\mathbb{C}$, the parameter $\ep$ considered small, and certain (natural) boundary conditions on $y$. A perturbative expansion of a solution to such an equation takes the general form
\begin{equation} \label{eq:intro_trans}
    y(z;\epsilon) = y^{(0)}(z;\epsilon) + \sum_i e^{-\chi_i(z)/\epsilon}y^{(i)}(z;\epsilon),
\end{equation}
where the $y^{(i)}(z;\epsilon)$ are divergent expansions similar to \eqref{eq:observable}. The above is thus a trans-series in a small parameter, $\epsilon$, and involves an additional holomorphic variable, $z \in \mathbb{C}$. Although \eqref{eq:intro_ode} is very standard, we emphasise that resurgent techniques have more typically been applied to the study of non-parametric differential equations. There are a number of related works by \emph{e.g.} Howls~\cite{howls2010exponential} (on trans-series for boundary-value problems) and Byatt-Smith~ \cite{byatt2000borel} (on the Borel transform), but we believe the approach of working entirely in the Borel plane for such singularly perturbed problems is not as well appreciated. 

Singularly perturbed differential equations such as \eqref{eq:intro_ode} are standard in applied mathematics, where problems may also take more obscure forms, including nonlinear equations, integro-differential equations, partial differential equations, and boundary-value problems. In many such problems, the emphasis is more to derive leading-order exponentially-small estimates---perhaps only the first approximation to, say $\e^{-\chi_1/\ep} y^{(1)}$ in \eqref{eq:intro_trans}. Surprisingly, these exponentially-small effects can dictate a number of key properties of the associated physical problem; applications in classical physics have included modelling of dendritic crystal growth~\cite{kruskal1991asymptotics}, Saffman-Taylor viscous fingering, water waves, transition to turbulence, vortex reconnection, pattern formation, and many others. The development of exponential asymptotics for such problems has been enormously successful. In some sense, the focus has mostly been on the trans-series series side of the correspondence \eqref{eq:correspondence}---that is, manipulation and analysis of the divergent series. 

In the methodology of Chapman \emph{et al.}~\cite{chapman_1998_exponential_asymptotics} for instance, one begins by studying the early terms of the asymptotic expansion, say $y_0(z)$ or $y_1(z)$, and noting that these early orders often contain singularities in $\mathbb{C}_z$ (poles or branch points). By the singular nature of the differential equation, subsequent terms, say $y_n$, depend on differentiation of the previous terms. By the linearity of this procedure, no new singularities are introduced beyond those that appear in the early terms. Thus, further  differentiation of the early terms can lead to factorial-over-power divergence as $n \to \infty$ and one posits that the late terms of $y^{(0)}(z)$ in \eqref{eq:intro_trans} satisfy
\begin{equation}
    y^{(0)}_n(z) \sim \frac{A(z)\Gamma(n+\gamma)}{\chi(z)^{n+\gamma}} \qquad \text{as $n \to \infty$}. 
\end{equation}
The components such as $A$, $\gamma$, and $\chi$ are derived using matched asymptotics in the complex plane, and optimal truncation and Stokes smoothing is applied in order to relate the above divergence to the leading-order trans-series correction, involving $\e^{-\chi/\ep}$. Such complex-plane asymptotics have been a standard approach since the seminal work of Kruskal \& Segur~\cite{kruskal1991asymptotics} and other similar applications can be found in the compendium by Segur \emph{et al.}~\cite{segur2012asymptotics}.


Though the above approach is powerful, recent research has revealed an increasing number of problems that stretch or challenge the conventional methodology of applied exponential asymptotics. These include, for instance, problems involving coalescing singularities~\cite{trinh2015exponential}, interacting Stokes lines~\cite{king1998interacting},  partial differential equations~\cite{chapman2005exponential,body2005exponential}, and higher-order Stokes phenomena. Our second goal in this work is to develop analogues of the factorial-over-power ansatz using aspects of resurgence; we shall argue that this provides some powerful intuition to the rich geometric structure of such problems. The approach has the potential to allow the creation of new model  problems in singular perturbation theory that exhibit exponential asymptotic effects---and for which physical intuition may be lacking. This is crucial, for instance, in applications to partial differential equations where techniques remain quite limited.

Our final goal is to provide important links between the disparate communities studying beyond-all-orders asymptotics. Unsurprisingly, due to both the ubiquity of perturbation theory in the physical sciences and the rich mathematical structure of resurgence, there is a wide range of researchers working in the closely related, but often disparate, fields of exponential asymptotics, singular perturbation theory, Borel resummation, and resurgence across the spectrum of applied mathematics, geometry/quantum field theory, and analysis---we have attempted to summarise and classify a selection of such works in Table \ref{table:summary}. The work grew out of the recent Isaac Newton Institute program on applicable resurgent asymptotics and we hope parts of the paper serve as a useful review to bridge some of the language barrier between these different communities in exponential asymptotics. More precisely, our goal is to make some steps towards unifying these approaches, whilst reviewing and making accessible some aspects of resurgence theory to the applied exponential asymptotics community. At the same time, we phrase the factorial-over-power ansatz method common in applied mathematicians in a language familiar to other practitioners of resurgence (such as those working in theoretical physics) and review some tools of singular perturbation theory. 

\subsection{Outline and summary of results}
In \cref{sec:background}, we review the background and terminology of resurgence, primarily for non-parametric asymptotic expansions. Then in \cref{sec:parametricborel}, we study properties of holomorphic trans-series and their corresponding parametric Borel transforms in generality. A new feature, as compared with constant trans-series, is the presence of boundary layers---we shall explain boundary layers in terms of the interplay between singularities in the Borel plane $\Gamma_w$ and singularities in the physical plane $\Gamma_z$. In applied exponential asymptotics, methodologies apply a heuristic called Van-Dyke's rule~\cite{vandyke_book} for the matching of inner and outer limits of the divergent series. We later introduce a purely holomorphic version of Van-Dyke's rule that relates the two expansions about points in $\Gamma_w$ and $\Gamma_z$.

In \cref{sec:singularperturbation}, we turn our attention to holomorphic trans-series that arise as solutions to singularly perturbed linear ordinary differential equations. In applied exponential asymptotics the celebrated `factorial-over-power' ansatz method (briefly outlined above) determines the (leading order) components of the trans-series side of \eqref{eq:correspondence}. We shall explain this ansatz in terms of the holomorphic side of the correspondence (\textit{i.e.} in terms of a singularity ansatz in the Borel plane), and thereby extend the method to determine all-orders of the trans-series on the left hand-side of \eqref{eq:correspondence}. In  \cref{sec:applications} we conclude with some examples, pathologies, and future directions.

We emphasise the following perspective: From the view of the Borel plane (the trans-series side of correspondence \eqref{eq:correspondence}), singularly perturbed problems are comprised of two parts: an `operator part' (outer) and an `initial data part' (inner). In terms of the operator part, we follow many of the same ideas as the communities studying the exact WKB method (cf. the monographs by Aoki \emph{et al.}~\cite{aoki2008virtual} and Honda \emph{et al.}~\cite{honda2015virtual}). In particular, we review how in the Borel plane, singularly perturbed problems yield partial differential equations, $\mathscr{P}_By_B = 0$, for the parametric resurgent function on the right-hand side of the correspondence \eqref{eq:correspondence}. We show that such 
\textit{partial} differential equations yield \textit{ordinary} differential equations for the holomorphic components of the trans-series (or equivalently $z$-dependent coefficients in the expansions near singularities in the Borel plane). 

However, the `operator part' does not provide initial conditions for these equations---in this way, singular perturbation theory furnishes us with a `blank template' trans-series. The `initial condition' (inner) part of the methodology fills in the blanks. In our work, we show that, by re-scaling near singularities $\Gamma_z$, we obtain an inner Borel \textit{ODE} operator, $\mathscr{P}_B^{\text{inner}}$, describing a constant ($0$-dimensional) trans-series problem of the type originally studied by \'Ecalle. Upon translating back to `physical space', inner equations may be ordinary differential equations with irregular singular points. This problem is `difficult' but solving this resurgent connection problem and finding the (infinitely many) associated Stokes' constants yields initial conditions at $\Gamma_z$ for the infinitely many ordinary differential equations for trans-series components---thereby solving the original singularly perturbed problem.

Finally, turning back to holomorphic trans-series more generally, we shall also discuss how modifications of the `factorial-over-power' ansatz of traditional exponential asymptotics \cite{chapman_1998_exponential_asymptotics} are necessary if singularities in the Borel plane differ from power law singularities. We demonstrate a result that bypasses the usual optimal truncation and Stokes switching argument and shows directly how asymptotics of power series coefficients may be related to Hankel integrals, which subsequently determine leading-order exponentially small terms. 

\section{Background}\label{sec:background}
We now introduce the background material necessary to understand certain resurgent aspects of singular perturbation theory. We begin with a re-framing of some well-known properties of holomorphic functions with discrete singularities. Next, we review the method of Borel resummation and elucidate the correspondence between trans-series and resurgent functions. The reviews by Aniceto \textit{et al.}~\cite{Aniceto:2018bis} and Dorigoni~\cite{dorigoni2019introduction} have been helpful in the preparation of some parts of this section and we adopt a similar physical approach to mathematical rigour throughout.

\subsection{Coefficients of holomorphic functions and power series} \label{subsec:complexanalysis}
Consider a locally convergent power series
\begin{equation} \label{eq:fser}
    f(x) = f_0 + f_1 x + f_2 x^2 + \ldots
\end{equation}
defined for $x\in\mathbb{C}$. In this paper, we shall equivalently refer to the above as a \emph{holomorphic germ} and write $f\in \mathcal{O}_0(\mathbb{C})$. The uniqueness of analytic continuation suggests that the coefficients, $\{f_n\}$, must necessarily contain essential information about the nature and location of singularities of $f$. The following discussion elucidates precisely how the singularity data is encoded and may be read directly from the asymptotics of $f_n$. Although this relates to some elementary facts of complex analysis, we  present the results in a way which will be convenient for later application to Stokes phenomena.

\begin{lemma}[Coefficients of a power series] \label{lemma:hankelcoefficients}
Let $f \in \mathcal{O}_0(\mathbb{C})$ be the germ about the origin, $x=0$, of a holomorphic function, which we shall write as \eqref{eq:fser}. Suppose that the analytic continuation of $f$ has a discrete set of singularities at $x = \chi_1, \chi_2, \chi_3 \ldots \in \mathbb{C}$ and has sub-exponential growth at infinity. The coefficients, $f_n$, of the germ of $f$ may then be given by the integral,
\begin{equation}\label{eq:hankelcoefficients}
    f_n = \frac{1}{2 \pi \i}\sum_{\chi} \int_{\mathcal{H}_{\chi}} \de{w} \, \e^{-w n} f(\e^w) \,,
\end{equation}
where the sum is taken over the singularities, $\chi = \chi_j$ for $j \geq 1$ of $f$, and $\mathcal{H}_{\chi}$ denotes a Hankel contour that encircles each respective singularity.
\end{lemma}
\begin{proof}
This lemma is a re-writing of Cauchy's integral formula. We write the coefficients as 
\begin{equation}
    f_n = \frac{1}{2 \pi \i} \oint_{S^1} \frac{\de{x}}{x^{n+1}} \, f(x) \,,
    \end{equation}
and then apply the conformal transformation $x = \e^{w}$. Deformation of the contour along the unit circle then yields the Hankel contours. Under this map, integration in the $x$-plane can then be expressed as integration on the $w$ cylinder. This is illustrated in figure \ref{fig:cylinder}.

\begin{figure}
    \centering
    \includegraphics{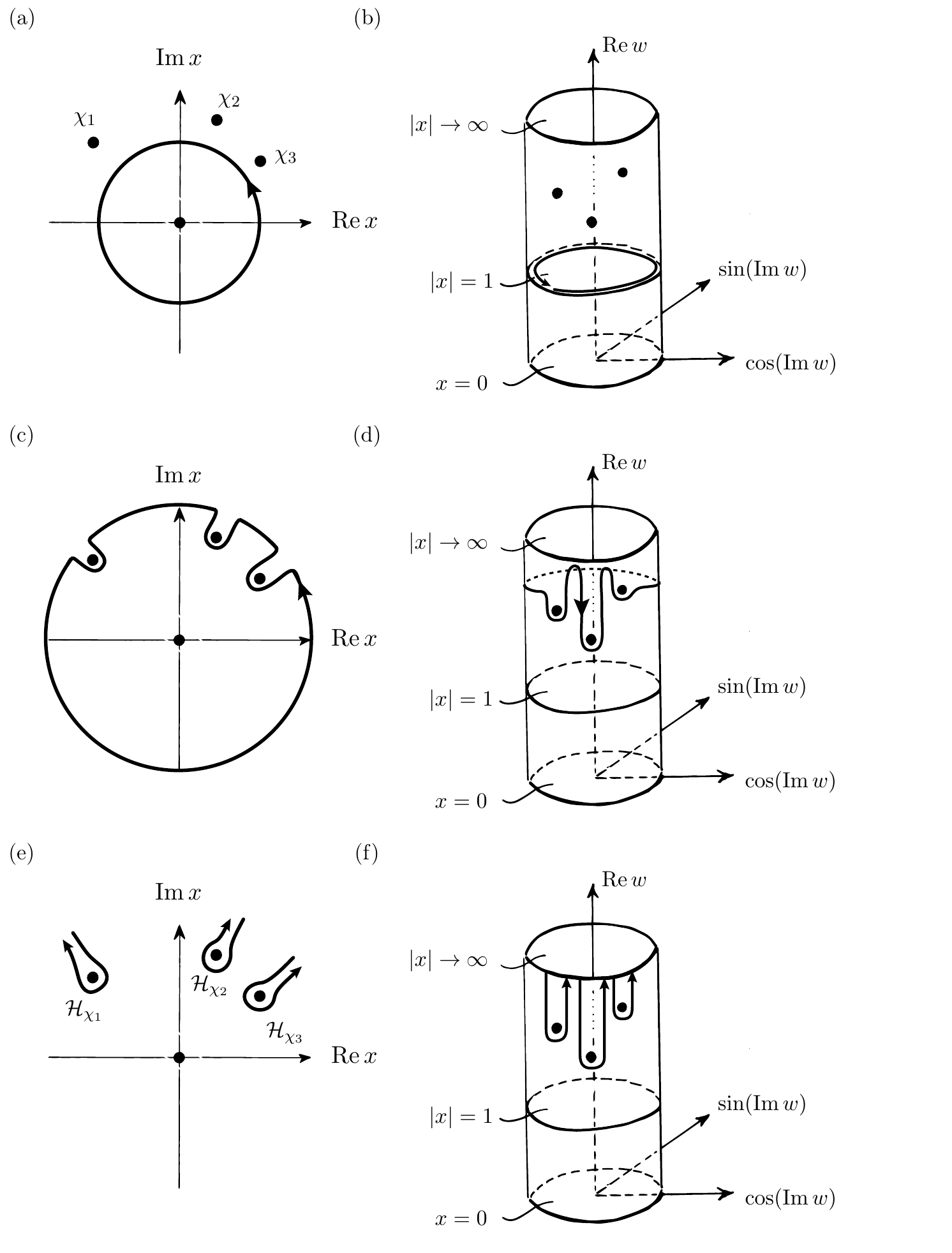}
    \caption{The power series coefficients, $f_n$, of \eqref{eq:fser} can be expressed as a sum of Hankel integrals, $\mathcal{H}_\chi$, around each singularity. This is shown via subplots (a, c, e) for an example with three singularities at $x = \chi_1$, $\chi_2$, and $\chi_3$. From the transformation $x = \e^{w} = \e^{\Re w} \e^{\im \Im w}$, the coefficients can then be expressed in terms of Laplace-type integrals \eqref{eq:hankelcoefficients}. The correspondence of points in $x$ to points in $w$ can be visualised via the cylindrical geometry shown in subplots (b, d, f). Together, this illustrates the sequence of necessary deformation to the Hankel contour in either coordinate system.}
    \label{fig:cylinder}
\end{figure}

\end{proof}
\noindent The above lemma demonstrates the relationship between the coefficients of a holomorphic germ and the singularities in the function's analytic continuation. A version of the classic Darboux theorem, which is notably described by Dingle~\cite{dingle1973asymptotic}, follows from formula \eqref{eq:hankelcoefficients}; this we now explain.

Firstly, let us suppose that the analytic continuation of $f$ has a single singularity at $x=\chi$. Suppose further that at this point the singularity in $f(x)$ takes the form of a power law
\begin{equation}\label{eq:elementarysingularity}
f(x) = (x-\chi)^{-\alpha}\underbrace{\Bigl(a_0 + (x-\chi)a_1 + (x-\chi)^2 a_2 + \ldots\Bigr)}_{h(x)} + \text{reg.}
\end{equation}
with $\alpha \in \mathbb{C} \backslash \mathbb{Z}_{\le 0}$. Since $x=\chi$ is the only singularity, then the regular part (reg.) and $h(x)$ are entire. We may evaluate the integral \eqref{eq:hankelcoefficients} at large $n$ with \eqref{eq:elementarysingularity} to determine the asymptotics of the coefficients. After a change of variables and scaling ($e^w = \chi e^{s/n}$) one may use results from appendix \ref{sec:gammaappendix} to evaluate the integral at large $n$, term-by-term, giving
\begin{equation} \label{eq:ffactpow_1}
    f_n = \frac{n^{\alpha-1}(-1)^\alpha}{(\chi)^{n+\alpha}} \left(\frac{a_0}{\Gamma(\alpha)} + \frac{1}{n} \left( \frac{a_0 \alpha(\alpha-1)}{2\Gamma(\alpha)} - \frac{(\chi)a_1}{\Gamma(\alpha-1)}\right) + \ldots \right) \,.
\end{equation}
For singularities of the type \eqref{eq:elementarysingularity}, instead of \eqref{eq:ffactpow_1}, the asymptotics may be alternatively expressed as $n \to \infty$ by
\begin{multline} \label{eq:ffactpow}
    f_n = \frac{\Gamma(n+\alpha)(-1)^{\alpha}}{(\chi)^{n+\alpha}\Gamma(n+1)} \biggl( \frac{a_0}{\Gamma(\alpha)} -
    \frac{1}{(n+\alpha-1)}\frac{\chi a_1}{\Gamma(\alpha-1)} \\
    + \frac{1}{(n+\alpha-1)(n+\alpha-2)}\frac{\chi^2 a_2}{\Gamma(\alpha-2)} - \ldots \biggr) \,,
\end{multline}
which follows from the binomial theorem. Note that one may begin to compare these two expansions using the result \eqref{eq:gammaasymptotics} from the appendix. 

In the context of applied singular perturbation theory, the leading factorial-over-power term in \eqref{eq:ffactpow} is convenient for application; indeed, it is the form that is preferred by many practitioners (e.g. \cite{chapman_1998_exponential_asymptotics}). However, as we see in the following (corollary~\ref{cor:stokesswitching}), when studying leading-order exponential corrections a late-terms ansatz of the form \eqref{eq:ffactpow_1} with a pre-factor of $n^{\alpha-1}$ is sufficient.

Now suppose that $f$ has additional power law singularities at $x = \chi_1,\chi_2, \ldots$ and near each $f$ can be written as a local expansion,
\begin{equation}
    f(x) = (x-\chi_i)^{-\alpha_i}(a_0^{(i)} + a_1^{(i)}(x-\chi_i) + \ldots) + \text{reg.}, \qquad \text{for $i = 1, 2, 3, \ldots$}
\end{equation}
In this case, the regular parts (reg.) have finite radii of convergence and consequently the large-$n$ term-by-term evaluation of the integral \eqref{eq:hankelcoefficients}, similar to the writing of \eqref{eq:ffactpow}, is only asymptotically valid. Indeed, we obtain a divergent \textit{trans-series} in the small parameter $1/n$ for the coefficients $f_n$, where more distant singularities are suppressed by an exponentially small pre-factor, $\chi^{-n-\alpha} = \exp(-(n-\alpha)\log \chi)$. This leads to the expansion in $1/n$ of the coefficients\footnote{Numerical factors present in \eqref{eq:ffactpow} have been absorbed into the coefficients $a_i^{(j)}$.}
\begin{equation} \label{eq:fn_trans}
    f_n = \sum_i \frac{1}{\chi_i^{n+\alpha_i}} \frac{\Gamma(n+\alpha_i)}{\Gamma(n+1)}\left( a_0^{(i)} + \frac{1}{(n+\alpha_i-1)}a_1^{(i)} + \ldots \right) \,.
\end{equation}

\begin{remark}
Whilst the above trans-series result \eqref{eq:fn_trans}, which applies to the case of power-law singularities, can be obtained using the binomial theorem, lemma \ref{lemma:hankelcoefficients} may be applied more generally---that is, there exists is a relationship between the asymptotics of a local expansion and the type of singularities on the boundary of convergence. In section \ref{subsec:generalsings} we discuss more general singularities.
\end{remark}

This section may be summarised as follows. The coefficients of an expansion about any given singularity can be found in the large-$n$ asymptotics of the coefficients about the origin (or indeed any other singular point). All singularities of $f$ are connected in this way. This may be relatively unsurprising when posed in the language of complex analysis (\textit{i.e.} analytic continuation is uniquely determined by the germ at the origin)---however, as we shall review in the following sections, these same results can be translated to facts about divergent asymptotic series via Borel resummation \cite{ecalle1981fonctions}. In this context, one observes the striking fact that \textit{divergent perturbative series often `know' about their own non-perturbative corrections}.

We conclude this section with a result closely related to lemma \ref{lemma:hankelcoefficients}. This expression will be crucial to the analysis of Stokes phenomena to follow. In brief, we may relate a certain Hankel contour integral of $f$, studied in the limit of $\ep \to 0$, with the asymptotics of the coefficients of $f$ expanded about a particular singularity. 

\begin{lemma}\label{lemma:hankel}
Suppose $f(w)$ satisfies the conditions of lemma \ref{lemma:hankelcoefficients} and let $\epsilon$ be a complex number of small modulus. Around a particular singularity, $w = \chi$, we have
\begin{equation}
I_f(\epsilon) := \int_{\mathcal{H}} dw \, e^{-w/\epsilon} f(w) 
    =\int_{\mathcal{H}} dw \, e^{-w/\epsilon} (w-\chi)^{-\alpha}(a_0 + a_1(w-\chi) + \ldots)\,,
\end{equation}
where $\mathcal{H}$ is a Hankel contour centred on $\chi$. The integral may then be evaluated term-by-term to give
\begin{equation} \label{eq:If_Lemma}
    I_f(\epsilon) = \frac{2 \pi i}{\epsilon^{\alpha-1}\Gamma(\alpha)}e^{-\chi/\epsilon}\left(a_0 + (\alpha-1)a_1 \epsilon + (\alpha-1)(\alpha-2)a_2 \epsilon^2 + \ldots\right) \,,
\end{equation}
in the limit $\ep \to 0$. Note that $I_{f}(\epsilon)$ will have a zero radius of convergence in $\epsilon$ if there is an additional singularity besides $w=\chi$ in the analytic continuation of $f$.
\end{lemma}
\begin{proof}
The result follows from term-by-term integration and use of appendix~\ref{sec:gammaappendix}. 
\end{proof}
Compared with the previous lemma~\ref{lemma:hankelcoefficients}, we see that the asymptotic expansion in $\ep$ of the Hankel integral,   $I_f(\epsilon)$, appearing in \eqref{eq:If_Lemma} is very closely related to the asymptotic expansion in $1/n$ of the power-series coefficients of $f$, which appears in \eqref{eq:ffactpow}. We will make clear in the following that this relationship is the basic idea behind resurgence. Above in \eqref{eq:If_Lemma}, note also that the leading (in $\epsilon$) part of the expression is the factor $2 \pi \i / \epsilon^{\alpha-1} \e^{-\chi/\epsilon}$ which is reminiscent of the term obtained by Berry's Stokes smoothing procedure in exponential asymptotics \citep{berry1988stokes} outlined in the introduction; this will be explained in more detail later (c.f. corollary \eqref{cor:stokesswitching}).

\subsection{Asymptotics and trans-series for non-parametric functions}\label{subsec:asymptoticsandtrans}
In the first section of this work, we are concerned with (non-parametric) algebraic asymptotic expansions of the form
\begin{equation} \label{eq:yser}
    y(\epsilon) = y_0 + y_1 \epsilon + y_2 \epsilon^2 + \ldots
\end{equation}
with zero radius of convergence in $\mathbb{C}_{\epsilon}$. We further restrict to so-called Gevrey-1 sequences which satisfy the bound $|f_n| \le A B^n n!$ for some constants $A$ and $B$ and refer to such sequences loosely as factorially divergent. We consider such asymptotic series as an element of the formal power series algebra $\mathbb{C}_1[[\epsilon]]$, where multiplication is pointwise power series multiplication. In the course of this work, we will find it necessary to extend this algebra with exponential symbols $e^{-1/\epsilon}$ and write asymptotic sequences formally as 
\begin{multline}\label{eq:constanttransseries}
    y(\epsilon) = \Bigl[y_0^{(0)} + y_1^{(0)} \epsilon + y_2^{(0)} \epsilon^2 + \ldots \Bigr] + \Bigl[e^{-\chi_1/\epsilon}\epsilon^{-\alpha_1}(y_0^{(1)} + y_1^{(1)} \epsilon + \ldots)\Bigr] \\ + \Bigl[ e^{-\chi_2/\epsilon}\epsilon^{-\alpha_2}(y_0^{(2)} + y_1^{(2)} \epsilon + \ldots) \Bigr] + \ldots
\end{multline}
with $\alpha_i, \chi_i, y_i^{(j)} \in \mathbb{C}$. Formal extended asymptotic sequences of this form are known as log-free, height-$1$ trans-series. Trans-series are closed under many operations; in particular, they are an exponentially-closed ordered differential field, which we denote by $\mathbb{C}_1[[[\epsilon]]]$. We refer the reader to the work of Edgar~\cite{edgar2010transseries} for a thorough pedagogical review of trans-series, we will make little use of the full machinery in the present work. We refer to trans-series of the form \eqref{eq:constanttransseries} as \textit{constant} trans-series in contrast to the series studied later where the components $\chi_i$ and $y_i^{(j)}$ will be functions of an additional holomorphic variable $z \in \mathbb{C}$.

\begin{remark}
Note that the power-series coefficient expansion of a holomorphic function, $f$, seen in \textit{e.g.}~\eqref{eq:fn_trans} yields such a trans-series with $1/n$ playing the role of $\epsilon$.
\end{remark}

\subsection{The Borel transform}\label{subsec:boreltransform}
We now begin our discussion of Borel resummation. The Borel transform, $\mathcal{B}$, is a way of associating a holomorphic germ to an element of $\mathbb{C}_1[[\epsilon]]$. For the moment, let us consider the case of the asymptotic expansion \eqref{eq:yser} with $y_0 \equiv 0$. We shall re-introduce the constant part in the following section. We define $\mathcal{B}:\epsilon \mathbb{C}_1[[\epsilon]] \to \mathcal{O}_0(\mathbb{C})$ by
\begin{equation} \label{eq:Bgerm}
    \mathcal{B}: (y_1,y_2,y_3,\ldots) \to \frac{y_1}{0!} + \frac{y_2}{1!} w + \frac{y_3}{2!} w^2 + \ldots + \frac{y_n}{(n-1)!} w^n + \ldots
\end{equation}
Thus the Borel transform, $\B$, divides-out the factorial divergence of an asymptotic sequence. Notice that the Gevrey-1 condition of the expansion \eqref{eq:yser} ensures that the germ \eqref{eq:Bgerm}, obtained in the above way, will have a non-zero (but not necessarily infinite) radius of convergence; put differently, if the Borel transform of \eqref{eq:yser} has finite radius of convergence, then the original expansion must be asymptotic. In a minor abuse of notation, we shall write 
\[
\mathcal{B}[y](w) \equiv y_B(w),
\]
to denote the holomorphic function obtained by the analytic continuation of the germ generated from $y(\epsilon)$. Subsequently, we may form a Riemann surface,\footnote{$\Sigma$ is the space of homotopy classes of open curves with fixed origin and supported in $\mathbb{C}\backslash \Gamma$. The complex structure on $\Sigma$ is then obtained by pulling back the complex structure on $\mathbb{C} \backslash \Gamma$ with the natural covering map $\Pi: \Sigma \to \mathbb{C}\backslash \Gamma$} $\Sigma$, associated to $y_B(w)$. We denote the set of singular points in the analytic continuation of $y_B$ by $\Gamma_w$.

The map $\mathcal{B}$ may be extended to account for $y_0$ as follows. First, note that the ring $\mathcal{O}_0(\mathbb{C})$ may be further endowed with an algebra structure with product
\begin{equation}\label{eq:starproduct}
    f_B \star g_B (w) = \int_0^w ds f_B(s)g_B(w-s) \,,
\end{equation}
the integral here encodes the formal power series convolution product. This convolutive algebra currently lacks a unit, which we rectify by adjoining a formal $\delta$-function as $\mathcal{O}_0(\mathbb{C}) \oplus \mathbb{C}\delta$ that satisfies $\delta \star f_B = f_B$. The Borel transform may then be extended to all asymptotic series, $\mathbb{C}_1[[\epsilon]]$, by defining, in addition to \eqref{eq:Bgerm} the rule, 
\begin{equation}
    \mathcal{B}: y_0 \mapsto y_0 \, \delta \,.
\end{equation}
Then $\mathcal{B}$ extends to a morphism with respect to the convolution product \eqref{eq:starproduct}. On a practical level, this means that performing the Borel resummation procedure discussed below, the $y_0$ term of a formal asymptotic series is left untouched by the integral transform. We refer the reader to the excellent work of Nikolaev \cite{nikolaev2020exact} for a more thorough review of rigorous aspects of Borel resummation and Gevrey-1 asymptotics.

\paragraph{Resurgent functions.}
One of the great achievements of \'Ecalle~\cite{ecalle1981fonctions} (see also the recent reviews by Sauzin~ \cite{sauzin2014introduction,sauzin2007resurgent}) is to construct an algebra of functions known as \textit{resurgent functions}, defined on the Riemann surface, $\Sigma$. 

A \textit{resurgent function} is a holomorphic function that admits endless analytic continuation. Informally, this means that it must be possible to analytically continue $y_B$ along any ray from the origin by avoiding a discrete set of singularities. The analytic continuation contains only\footnote{In fact, \'Ecalle's original work \cite{ecalle1981fonctions} discusses only so-called simple singularities, i.e. singularities with $\alpha=-1$, together with logarithmic singularities. Generalisations were later made by the same author in \cite{ecalle1992introduction}. In line with typical examples studied in the applied mathematics literature, we exclude logarithmic singularities.} power law singularities of the form \eqref{eq:elementarysingularity}. 

When one restricts to the class of resurgent functions we obtain (in the way outlined above) an algebra, denoted $\mathcal{H}(\Sigma)$, such that the Borel transform is an algebra morphism. This algebra is referred to as the \textit{convolutive model}. The pull-back (pre-image) of $\mathcal{H}(\Sigma)$ by $\mathcal{B}$, denoted $\hat{\mathcal{H}}(\Sigma)$, is the algebra of resurgent asymptotic sequences---referred to as the \textit{multiplicative model}---and this is the place we work throughout this work.

\paragraph{Inverse Borel transform and Borel resummation.}
Above, we have seen how we may associate a holomorphic function, $y_B(w)$, to a divergent series, $y(\epsilon)$, via the Borel transform. We now review the definition of the inverse Borel transform, the asymptotic evaluation of which returns a divergent trans-series. Let $l_{\theta}$ be a ray in the complex plane, $\mathbb{C}_w$, emanating from the origin in the direction $\theta$. The \textit{inverse Borel transform}, $\B^{-1}$, in the direction $\theta$ is defined by
\begin{equation}\label{eq:inverseborel}
    Y_{\theta}(\epsilon) := \B^{-1}[y_B](\ep) = \int_{l_{\theta}} \de{w} \,  \e^{-w/\epsilon} \, y_B(w) \,. 
\end{equation}
We say that $y(\epsilon)$ is \emph{Borel summable} in the direction $\theta$ if this integral converges to a holomorphic function.

Now, suppose that $y_B$ is Borel summable in the direction $\theta=0$. From the definition of the Gamma function via \eqref{eq:gammafunction},  we may compute the asymptotics of $Y(\epsilon) := Y_{\theta=0}$ to be
\begin{equation}
    Y(\epsilon) = \epsilon \left(\frac{y_1}{0!}\right) + \epsilon^2 \left(\frac{y_2}{1!}\right) + 2! \epsilon^3 \left(\frac{y_3}{2!}\right) +  \ldots,
\end{equation}
which may be compared with $y(\ep) - y_0$ in \eqref{eq:yser}. Hence, we may regard the Gamma function integral as re-introducing the factorial divergence from $y_B$ to $y$. This procedure of taking a divergent series, $y(\epsilon)$, associating to it a holomorphic germ (or power series), $y_B$, with nice analytic continuation properties, and then writing down a holomorphic function, $Y(\epsilon)$, with the same asymptotics as $y - y_0$ is known as \textit{Borel resummation}. In Figure~\ref{fig:borelresum}, we illustrate the procedure and relationship between divergent series, the Borel transform, and the inverse Borel transform.

\begin{figure}[htb]
    \centering
    \includegraphics{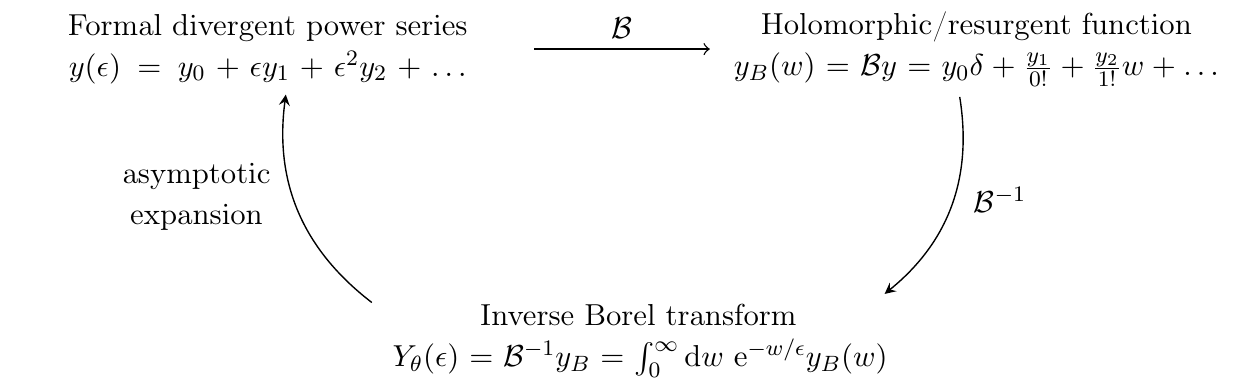}
    \caption{An illustration of the relationships between divergent series in powers of $\epsilon$, with the Borel transform, $y_B$, and subsequently with the inverse Borel transform, $Y_\theta$.}
    \label{fig:borelresum}
\end{figure}

\paragraph{Replacement rules for $y_B$.}
Above, we have explained the definition of the Borel resummation procedure of a given divergent series. Suppose now that $y(\epsilon)$ arises as the perturbative solution to a differential equation (which may include, for example, linear, non-linear or difference terms).
Rather than solving the equation directly in physical space $\mathbb{C}_{\epsilon}$, we may seek a solution of the form
\begin{equation} \label{eq:y_invB}
    y(\epsilon) = y_0 + \int_{(0,\infty)} \de{w} \,\e^{-w/\epsilon}y_B(w),
\end{equation}
so now the Borel transform, $\B y = y_B(w)$, is determined without explicit manipulation of the perturbative expansion for $y$. In order to determine $y_B$, the differential equation in $y$ is transformed to a differential equation for $y_B$. Notice that \eqref{eq:y_invB} is the inverse Laplace transform of $y_B(w)$ with respect to $\eta=1/\epsilon$ and hence the following is familiar from standard Laplace transform theory:

\begin{lemma}[Replacement rules]\label{lemma:replacementrules}
Given a differential equation for $y = y(\ep)$, a governing equation for $y_B$ is derived using the following observations
\begin{equation} \label{eq:reprules01}
\mathcal{B}\left[\frac{1}{\ep} y\right] = -y_B'(w), \qquad 
\mathcal{B}\left[\ep^2 \dd{y}{\ep}\right] = -w y_B(w), \qquad 
\mathcal{B}\left[y(\ep)^2\right] = y_B \star y_B(w).
\end{equation}
We shall refer to the above relations as `replacement rules', and these yield the governing equation for the Borel transform $y_B(w)$.
%

\end{lemma}
\begin{proof}
These follow directly from the definition of the inverse Borel transform \eqref{eq:inverseborel}. For instance, we write $\B^{-1}[-y_B']$, apply formula \eqref{eq:inverseborel}, and integrate by parts to derive the first rule of \eqref{eq:reprules01}.
\end{proof}

Later we shall also discuss the Borel transform of difference equations. In this case, it is easier to write expressions in terms of $\ep = 1/\eta$ with $y(\ep) = y(1/\eta) \equiv f(\eta)$. Then for instance, one may verify that 
\begin{equation} \label{eq:rep_discrete}
    \mathcal{B}[f(\eta+1)] = \e^{-w}y_B(w).
\end{equation}

\subsection{Stokes phenomena}
\label{subsec:stokesphenom}

When computing the inverse Borel transform, \eqref{eq:inverseborel}, there may be directions $l_{\vartheta}$ along which the integrand, $y_B(w)$, is singular. As one varies $\epsilon \in \mathbb{C}_\epsilon$ the convergence of the integral requires us to adjust the contour in $\mathbb{C}_w$ with $\arg \epsilon$. If the contour $l_\vartheta$ reaches a singularity of $y_B$ then, in order to keep the integration class constant (i.e. the inverse Borel transform continuous), we pick up a contribution from a Hankel contour around the singularity. This procedure is shown in \cref{fig:hankeljump} and an example is given later in \cref{subsec:irregularsingularpoints}. 

\begin{figure}[htb]
    \centering
    \includegraphics[scale=0.3]{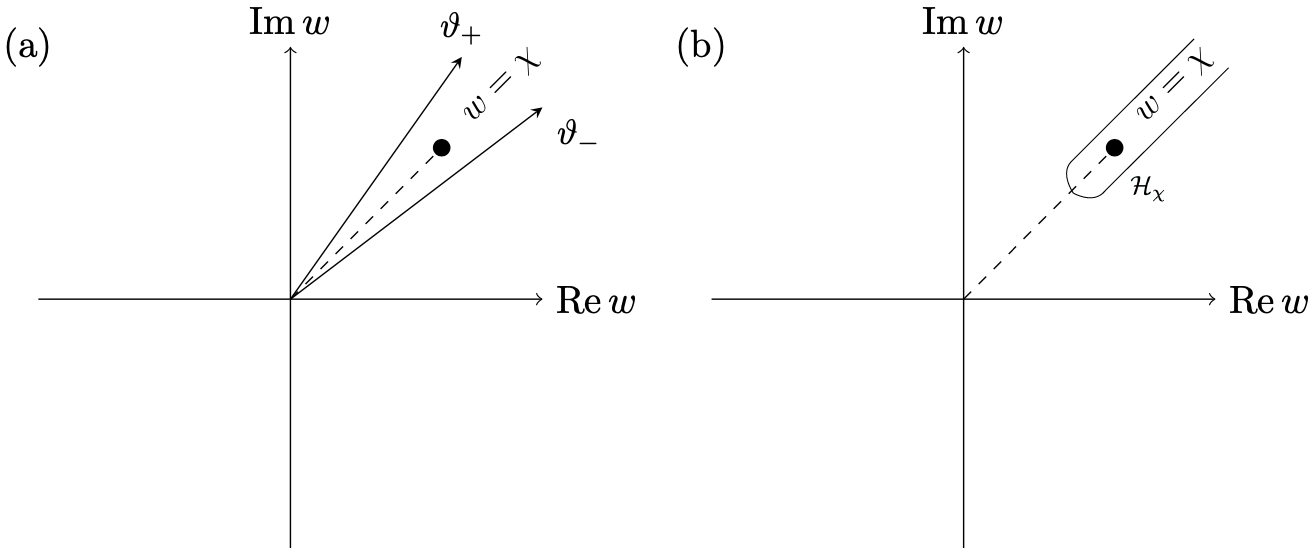}
    \caption{(a) A singularity lies at $w= \chi$. As $\ep$ varies in $\mathbb{C}_\ep$, it is necessary to adjust the integration contour $l_\vartheta$. If the integration contour crosses $\chi$, we include the contribution from integration along $l_{\vartheta_+} - l_{\vartheta_-}$. This yields the Hankel contour shown in (b).}
    \label{fig:hankeljump}
\end{figure}

In more detail, suppose that $y_B(x)$ has a singularity at $w=\chi$ with $\arg \chi = \vartheta$, then we see that
\begin{equation} \label{eq:tmp1}
    Y_{\vartheta_+}(\epsilon) - Y_{\vartheta_-}(\epsilon) = \int_{\mathcal{H}_{\chi}} \de{w} \, \e^{-w/\epsilon} y_B(w) \,.
\end{equation}
Hence the resummation of the divergent series is ambiguous in this direction. Suppose further that the singularity is of power law form so that
\begin{equation}
    y_B(w) = (w-\chi)^{-\alpha}(a_0 + a_1(w-\chi) + \ldots)  \,,
\end{equation}
then the integral \eqref{eq:tmp1} may be evaluated using lemma \ref{lemma:hankel} to give
\begin{equation}\label{eq:jump}
    Y_{\vartheta_+}(\epsilon) - Y_{\vartheta_-}(\epsilon) = \frac{2 \pi \i}{\epsilon^{\alpha-1}} \e^{-\chi/\epsilon}\left(\frac{a_0}{\Gamma(\alpha)} + O(\epsilon) \right) \,.
\end{equation}
The line $l_{\vartheta} \subset \mathbb{C}_{\epsilon}$ is then called a \textit{Stokes line} and the corresponding jump in the asymptotics as described by equation \eqref{eq:jump} is called \textit{Stokes phenomena}. In particular,  we see that across $l_{\vartheta}$ the asymptotics of $y_0 + Y(\epsilon)$ jump\footnote{Berry \cite{berry_1989} shows that this transition is smooth by evaluating the asymptotics of the integral within a different, boundary layer, regime. This is known as Stokes' smoothing.} according to
\begin{equation}
    y_0 + Y(\epsilon) \to \Bigl[ y_0 + \epsilon y_1 + \ldots \Bigr] +  \frac{1}{\epsilon^{\alpha-1}}e^{-\chi/\epsilon}(c_0 + \ldots), 
\end{equation}
for the constant $c_0$ specifically given by the pre-factors in \eqref{eq:jump}. In general, one may have multiple singularities on $\mathbb{C}_w$, and the corresponding collection of Stokes lines in $\mathbb{C}_{\epsilon}$ form a \textit{Stokes graph}.

\paragraph{Late terms.}
In applied exponential asymptotics approaches, it is common to analyse the late terms of perturbative solutions to differential equations \cite{boyd1999devil,chapman_1998_exponential_asymptotics,segur2012asymptotics}. We discuss how the trans-series structure of late-terms relates to (it is identical up to numerical constants) the trans-series structure of a divergent series $y(\epsilon)$. We first reinterpret Lemma~\ref{lemma:hankelcoefficients}. Consider $n$ a large parameter and write
\begin{equation}
    f(x) = f_0 + f_1 x + f_2 x^2 + \ldots + f_n x^n + \ldots
\end{equation}
then writing $\mathcal{B}[f_n]$ for the Borel transform of the function $f_n$ we see that Lemma~\ref{lemma:hankelcoefficients} takes the schematic form\footnote{There is now no `perturbative' contour and we only integrate over the Hankel contours $\mathcal{H}_\chi$.}
\begin{equation}
    \mathcal{B}[f_n](w) = \B[f](e^{w}) \,,
\end{equation}
Namely, the Borel transform of the coefficients may be obtained by a simple conformal map of $f(x)$. In more detail, given a perturbative series $y(\eta)$ with Borel sum $f(x)$, then the trans-series structure of $y(\eta)$ is the same (up to a conformal map) as the trans-series structure of the late terms $f_n$ of the Borel transform $f(x)$

\begin{example}
We can verify this in an example. Consider the following differential equation\footnote{We suppose the boundary conditions are such that the inverse Borel transform of $y_B$ with $\gamma=(0,\infty)$ satisfies them.} for $f(\eta) = y(\ep)$, 
\begin{equation}
-2 \eta f''(\eta) + (6 \eta + 6)f'(\eta) + (4 \eta + 10) f(\eta) = 0 \,.
\end{equation}
The use of the replacement rules and Lemma~\ref{lemma:replacementrules} indicates that the Borel transform, $y_B$, satisfies
\begin{equation}\label{eq:coefficientexample}
    2(1-w)(2-w)y_B'(w) - (2w-3)y_B(w) = 0 \,.
\end{equation}
Seeking a power series solution $y_B(w) = c_0 + c_1 w + c_2 w^2 +\ldots $ we find the recurrence relation
\begin{equation}
    4(n+1)c_{n+1} +(3-6n)c_n + 2(n-2)c_{n-1} = 0 \,, \quad n>2\,.
\end{equation}
We may now take the Borel transform again in order to determine $c_n$. From the discrete rule \eqref{eq:rep_discrete} the Borel transform of the coefficients, $c_B(w)$, solves
\begin{equation}
    c_B'(w)(4e^{-w}-6+2e^w) + c_B(w)(3-2 e^w) = 0\,,
\end{equation}
This may be integrated to
\begin{equation}
    c_B(w) \propto \left((e^w-2)(e^w-1)\right)^{1/2} \,.
\end{equation}
If one solves the first-order differential equation \eqref{eq:coefficientexample} then we immediately verify
\begin{equation}\label{eq:coefficientexampleconclusion}
    c_B(w) = \B[c_n](w) \propto y_B(e^w) = \B[y](\e^w) \,.
\end{equation}
Note that, according to lemma \ref{lemma:hankel}, the contour in the inverse Borel transform of $c_B$ should be a sum of Hankel contours around the singularities at $x=\log 2$ and $x = 0$ in order to recover $c_n$.
\end{example}
The motto of this subsection is as follows. Suppose that $y(\epsilon)$ is a perturbative solution to a differential equation. To understand the trans-series structure of $y(\epsilon)$ it is equivalent to work with the $1/n$ trans-series structure of the late terms in the expansion of $y(\epsilon)$ or to work with the $\epsilon$ trans-series structure of $y(\epsilon)$ directly. Later we will take a third viewpoint and show how, for holomorphic trans-series, the same data is encoded in a parametric holomorphic function.

\begin{example}\label{subsec:irregularsingularpoints}
The machinery of Borel resummation is a powerful tool to understand solutions of linear ordinary differential equations about irregular singular points. About such points, power-series solutions are divergent and resummation methods may be employed. We refer the reader to the comprehensive three-volume work by Delabaere \textit{et al.}~\cite{delabaere2016divergent} for an excellent review of resummation techniques. Later we will compare to the following example to see the subtle differences in which Stokes lines are interpreted in singularly perturbed problems that are the main focus of the present work.

A canonical example of a linear first-order equation with an irregular singular point is Euler's equation, given by
\begin{equation}\label{eq:euler}
    \epsilon^2 \dd{y}{\ep} +  y = \ep,
\end{equation}
with the boundary condition $y(\epsilon) \to 0$ as $\ep \to \infty$. The equation has an irregular singular point at $\epsilon=0$ and so the power series solution is only asymptotic in a sectorial domain in $\mathbb{C}_\ep$ with corner at $\epsilon = 0$. Initially, we shall consider \eqref{eq:euler} as defined along the positive real $\ep$-axis, and then consider the solution as $\ep$ varies.

Using Lemma~\ref{lemma:replacementrules}, we Borel transform \eqref{eq:euler}. This yields $y_B = 1/(w + 1)$, which has singular set $\Gamma_w = \{-1\}$ and is Borel summable everywhere except along the negative real axis. We may then recover a locally holomorphic solution to the ODE by substituting $y_B$ into the integral representation \eqref{eq:y_invB}. In general, for $\ep > 0$, the contour $\gamma$ may be composed of $\gamma_0 = (0,\infty)$ and the additional loop, $\gamma_1$, around $\Gamma = \{-1\}$. Thus we may write the general solution as
\begin{equation}
    y(\epsilon) = \int_{\gamma_0 = (0, \infty)} \de{w} \,\frac{\e^{-w/\epsilon}}{1+w} + c \int_{\gamma_1} \de{w}\, \frac{\e^{-w/\epsilon}}{1+w} \,.
\end{equation}
Imposing the boundary condition of $y \to 0$ as $\ep \to \infty$, then we find that $c=0$. Now, we may analytically continue the integral onto $\mathbb{C}_\ep \backslash \mathbb{R}_{<0}$ by adjusting the contour appropriately as $\epsilon$ is varied to $\gamma_0 = \e^{i \text{arg}(\ep)}(0,\infty)$. Along the negative real axis we see that the solution jumps according to
\begin{equation}
    y_{+} - y_{-} = \int_{\mathcal{H}_{-1}} \de{w} \, \frac{\e^{-w/\epsilon}}{1+w}.
\end{equation}
The Hankel contour in this case picks up the residue of the simple pole at $w=-1$ so asymptotically,
\begin{equation}
\begin{split}
    y_-(\epsilon) &= \epsilon - \epsilon^2 + 2! \epsilon^3 - 3! \epsilon^4 + \ldots + 2 \pi i \e^{1/\epsilon} \,, \\
    y_+(\epsilon) &= \epsilon - \epsilon^2 + 2! \epsilon^3 - 3! \epsilon^4 + \ldots \,,
\end{split}    
\end{equation}
and the asymptotic expansions on either side of the Stokes line $l = (-\infty,0)$ differ by an exponentially small term. The relevant Stokes graph is then as shown in figure \ref{fig:euler}.
\end{example}

\begin{figure}[htb]
    \centering
    \includegraphics{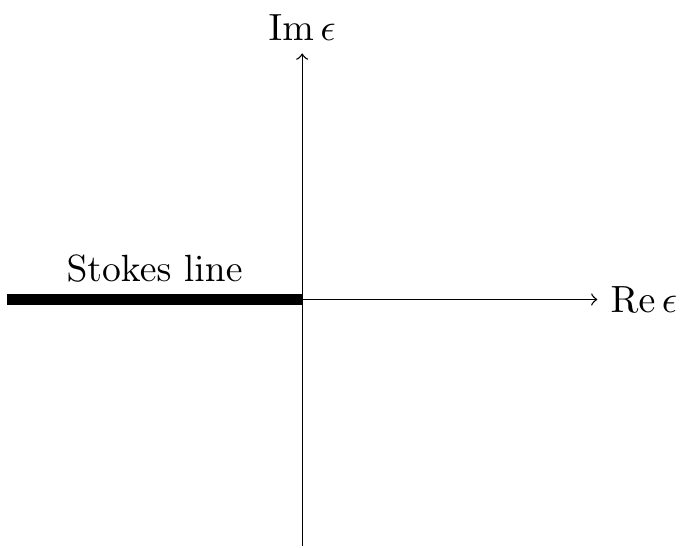}
    \caption{Stokes lines for Euler equation \eqref{eq:euler}. The asymptotic expansions on either side of the Stokes line differ by an exponentially small term, $2\pi \im \e^{1/\ep}$.}
    \label{fig:euler}
\end{figure}

\begin{table}[!htbp]\footnotesize
\begin{tabular}{p{6cm} p{8cm}}
\toprule
{\scshape school/area} & 
{\scshape references and selection of examples}  \\[1em]

\textbf{Hyperasymptotics}
& Factorial-over-power divergence \cite{dingle1973asymptotic};  Stokes' line smoothing \cite{berry_1989}; Hyperasymptotics \cite{berry1990hyperasymptotics,daalhuis_1993}; hyperterminants \cite{olde1996hyperterminants, olde1998hyperterminants}. Stokes' smoothing \cite{berry1988stokes, mcleod1992smoothing}. See  \cite{boyd1999devil} for a comprehensive review of superasymptotics and hyperasymptotics. \\[0.5em]

\textbf{Borel summation in DEs} \newline
The use of Borel resummation methods in the study of (often non-linear) differential equations. Painleve often serves as a tractable non-linear set of examples. 

& Comprehensive monograph of Costin \cite{costin2008asymptotics}; proof of Borel summability of rank one ODEs \cite{costin608408borel}; Tronque\'e solutions of Painlev\'e \cite{costin2015tronquee}; Borel summability of linear meromorphic ODEs \cite{balser1991multisummability}; Differential Galois theory monograph~\cite{kolchin1973differential}. \\[0.5em]

\textbf{Resurgence and alien calculus} \newline
Resurgence theory defines an algebra of resurgent functions and elucidates the self-similar nature of perturbative and trans-series. The theory develops the structure of Stokes' phenomena in terms of the Alien calculus.
& \'Ecalle's seminal work \cite{ecalle1981fonctions}. Reviews with physics focus \cite{Aniceto:2018bis,dorigoni2019introduction}. Excellent three volume series \cite{delabaere2016divergent}. Comprehensive lecture notes \cite{sauzin2014introduction} \\[0.5em]

\textbf{Geometry of Stokes' phenomena} \newline
The study of the monodromy structure of meromorphic connections with (ir-)regular singularities over algebraic curves. 

& Boalch's work on the topology of Stokes phenomena \cite{boalch2019topology}. Wild character varieties \cite{boalch2014geometry,boalch2015twisted}. Overview of Deligne's work on the Riemann-Hilbert correspondence \cite{katz1976overview} (Deligne's original work \cite{deligne2006equations}). Stokes structures \cite{sibuya1977stokes} (see also review by Sabbah \cite{sabbah2012introduction}). \\[0.5em]

\textbf{Applied exponential asymptotics} \newline
The use of the factorial-over-power ansatz to understand beyond-all-orders phenomena in the (mostly classical) physical sciences.
& See the compendiums by Segur \emph{et al.} \cite{segur2012asymptotics} and Boyd \cite{boyd1999devil} for a review of many applied problems in exponential asymptotics. Factorial-over-power methodology by Chapman \textit{et al.}~\cite{chapman_1998_exponential_asymptotics}; Saffman-Taylor fingering by Tanveer~\cite{tanveer1987analytic}. Applications to PDEs \cite{chapman2005exponential}. Free-surface flow problems with coalescing singularities. \cite{trinh2015exponential} \\[0.5em]

\textbf{Exact WKB analysis} \newline
Borel summability of the WKB approximation in the study of Schr\"odinger equations.

& Aoki-Kawai-Takei rigorously construct Borel transformed WKB solutions and explain the Vorus connection formula \cite{aoki1991bender}. Applications of microlocal analysis to the study of WKB Borel sums \cite{aoki1993microlocal}. See the comprehensive review monograph \cite{kawai2005algebraic}. Recent rigorous treatment of Borel summability \cite{nikolaev2020exact}. \\[0.5em]

\textbf{Applications in theoretical HEP} \newline
Resurgence and exact WKB analysis has recently been applied in the study a wide range of non-perturbative phenomena, and their associated geometrical structures, in gauge and string theory.
& A selection of examples: Spectral networks \cite{Gaiotto:2012rg}, ODE-IM correspondence \cite{Dorey:1999uk}, quantum curves and topological string theory \cite{Grassi:2014cla}, quantum knot invariants and Chern-Simons theory \cite{Garoufalidis:2020nut, gukov2016resurgence}. Gromov-Witten theory See comprehensive review \cite{Aniceto:2018bis}. \\

\bottomrule
\end{tabular}
\caption{A selection of references spanning different areas of beyond-all-orders asymptotics}\label{table:summary}
\end{table}

\section{Parametric Borel resummation} \label{sec:parametricborel}

In the previous section, we reviewed (non-parametric) asymptotic sequences, their Borel transform, and the associated resurgent properties. The corresponding theory for parametric asymptotic sequences of the form
\begin{equation}\label{eq:parametricsequence}
    y(z;\epsilon) = y_0(z) + \epsilon y_1(z) + \epsilon^2 y_2(z) + \ldots
\end{equation}
where now $z\in\mathbb{C}$, may be developed similarly but now such sequences exhibit new features due to the possibility of distinguished limits involving $z$ and $\ep$. In this section, we shall review the application of the Borel transform to sequences such as \eqref{eq:parametricsequence}, and discuss the notion of boundary layers in parametric trans-series.

Analogously to \cref{subsec:asymptoticsandtrans} we have a notion of Gevrey-1 asymptotics for \eqref{eq:parametricsequence} whereby we require that $y(z;\epsilon)$ is Gevrey-1 in $\epsilon$ for each $z \in \mathbb{C}_z$ away from a discrete set of singularities.  We denote the algebra of such sequences by $\mathbb{C}_1[[\epsilon]](z)$. We further restrict to resurgent sequences (that is, for fixed $z\in \mathbb{C}_z$ we suppose the resulting asymptotic sequence is resurgent in the sense of section \ref{subsec:boreltransform}). 

Analogous to \eqref{eq:Bgerm}, given a series of the form \eqref{eq:parametricsequence} we may define a parametric Borel transform where $z \in \mathbb{C}_z$ is a (holomorphic) parameter
\begin{equation}\label{eq:parametricboreltransform}
    \mathcal{B}: \left(y_1(z),y_2(z),\ldots \right) \to y_B(w,z) := \sum_{n=0}^{\infty} \frac{y_{n+1}(z)}{n!}w^n \,.
\end{equation}
The Borel transform $y_B(w,z)$ may now be considered a function on (possible covers of) $\mathbb{C}_z \times \mathbb{C}_w$. Note that $w \in \mathbb{C}_w$ is the distinguished variable here, since we will integrate it out to obtain a perturbative series. For fixed $z \in \mathbb{C}_z$ we may consider the analytic continuation of $y_B(w,z)$ which we again suppose has power law singularities at $w=\chi_i(z)$ with local powers $\alpha_i$---we denote this singular set by $\Gamma_w(z)$. Later, to be defined in \cref{subsec:boundarylayers}, we will also make use of a singular set $\Gamma_z$ associated to the physical plane $\mathbb{C}_z$ that roughly corresponds to singularities of early perturbative terms. In contrast with \cref{sec:background}, the curve $\Sigma_z$ associated to $y_B(w,z)$ now depends on $z \in \mathbb{C}_z$: the singularity locations and local expansions all depend on an additional holomorphic parameter.

We may similarly define a holomorphic function that recovers the asymptotics of $y(z;\epsilon)$ by the inverse Borel transform (cf. \eqref{eq:inverseborel})
\begin{equation}\label{eq:parametricborel_inversion}
    y(z;\epsilon) = y_0(z) + \int_{(0,\infty)} \de{w} \, \e^{-w/\epsilon} y_B(w,z) \,.
\end{equation}
We note that any parametric resurgent function $y_B(w,z)$ will define a (asymptotic) parametric series via this relation.

\subsection{Stokes lines in parametric series}
Stokes lines now arise in a subtly different way to the Stokes lines discussed in \cref{subsec:stokesphenom}. We now consider a fixed contour\footnote{The particular choice $\gamma =(0,\infty)$ is a consequence of the fact that we typically consider $\epsilon$ to be real---otherwise the perturbative part of the contour may be appropriately rotated.} $\gamma = (0,\infty)$ for the inverse transform \eqref{eq:parametricborel_inversion} and consider varying $z \in \mathbb{C}_z$. Along the following locus in $\mathbb{C}_z$:
\begin{equation}
    l_{\chi}: \Im \chi(z) = 0 \, \quad \Re \chi(z) > 0 \,,
\end{equation}
the contour, $\gamma$, intersects a singularity in $\Gamma_w(z)$ and, in order to keep the integration class constant, the contour is deformed to include a Hankel contour about the singularity (this is \textit{Stokes Phenomenon}). The union of all such $l_{\chi}$ define a \textit{Stokes graph} $l = \bigcup l_{\chi}$ on $\mathbb{C}_z$ as opposed to a Stokes graph on $\mathbb{C}_{\epsilon}$ as in \cref{subsec:stokesphenom}. Crossing such a locus thus adds exponential terms to the trans-series expansion. 

Let us introduce parametric dependence into \cref{lemma:hankel} and restrict to the case of a power law singularity at $w = \chi(z)$ for $y_B$ with power $-\alpha_\chi$. Then from \eqref{eq:parametricborel_inversion} we that across a Stokes line, $l_{\chi}$, the following term enters the trans-series
\begin{multline} \label{eq:stokeswitch_parametric}
\int_{\mathcal{H}_i} \de{w} \, \e^{-w/\epsilon}  (w-\chi(z))^{-\alpha_{\chi}}(a_0^{\chi}(z) + (w-\chi(z))a_1^{\chi}(z) + \ldots ) \\
    = \frac{2 \pi \i \ \e^{-\chi(z)/\epsilon}}{\epsilon^{\alpha_{\chi}-1}}(a_0^{\chi}(z) + \epsilon a_1^{\chi}(z) + \ldots)\,,
\end{multline}
where switching the integral and summation generally gives an asymptotic expansion on the right-hand side. In the applied exponential asymptotics literature when studied in the context of solutions, $y(z; \ep)$, of singularly perturbed differential equations, the derivation of this expansion in $\epsilon$ is referred to as \textit{Stokes switching}. The result is often arrived at by the Stokes smoothing arguments of Berry \cite{berry1988stokes} where properties of the governing (possibly non-linear \cite{chapman_1998_exponential_asymptotics}) differential equation are used to compute the jump across a boundary layer close to the Stokes line. For the leading-order exponentially small correction, the argument may be succinctly summarised by the following:
\begin{corollary}\label{cor:stokesswitching}
Suppose that a perturbative sequence \eqref{eq:parametricsequence} possesses the late-term asymptotics,
\begin{equation} \label{eq:cor_late}
    y_{n+1}(z) \sim \frac{n^{\alpha-1}n! (-1)^{\alpha}a(z)}{\chi(z)^{n+\alpha} \Gamma(\alpha)}, \qquad \text{as $n \to \infty$}\,,
\end{equation}
and the Borel transform has an associated power law singularity. Then, to leading order, the trans-series expansions of the Borel sum of the $y_n(z)$ differ by the following term across Stokes lines:
\begin{equation} \label{eq:cor_expswitch}
    \sim \left[\frac{2 \pi \i a(z)}{\epsilon^{\alpha-1}\Gamma(\alpha)}\right] \e^{-\chi(z)/\epsilon} \,.
\end{equation}
\end{corollary}
\begin{proof}
The perturbative terms, $y_n(z)$, of a sequence are related to the germ of the Borel transform by  \eqref{eq:parametricboreltransform}. We compare the leading-term with $n \to \infty$ in \eqref{eq:ffactpow_1} of \cref{lemma:hankelcoefficients}, and the leading-term with $\ep \to 0$ in \eqref{eq:If_Lemma} of \cref{lemma:hankel}. Hence, from the late terms in \eqref{eq:cor_late}, the Stokes switching \eqref{eq:cor_expswitch} follows by sending $1/n$ to $ \ep$, then $\chi(z)$ to $\e^{\chi(z)}$ via the conformal map of \cref{lemma:hankelcoefficients}, and finally noting the appearance of $2\pi\i$ due Cauchy's coefficient integral.
\end{proof}

\begin{remark}
It should be noted that, unlike often used presentations of Berry's Stokes smoothing argument, the above result does not require the specification of a singularly perturbed differential equation (or an analogous problem) yielding $y$---it requires only the late-term specification of the asymptotic sequence; then the assumption of Borel summability gives the exponential switching in \eqref{eq:cor_expswitch}.
\end{remark}

The knowledge of Borel singularities allows us to determine exponentially small terms (and corrections) that enter the trans-series across Stokes lines. We now generalise the above corollary to go beyond power law singularities in the Borel plane and explain the relationship between exponentially small terms and coefficient asymptotics at leading order. 
The key observation is that the Hankel integrals that arise from crossing Stokes lines are associated with \cref{fig:cylinder}(a) while the coefficient integrals are associated with the cylinder of \cref{fig:cylinder}(b). Thus, roughly speaking, the trans-series structure is the same up to this conformal map between the plane and the cylinder. In more detail, we see that crossing Stokes line associated to $w=\chi(z)$ we pick up the following contribution
\begin{equation}\label{eq:hankeleg}
    \int_{\mathcal{H}_\chi} dw \,e^{-w / \epsilon}f(w) 
    = \int_{\mathcal{H}_{\log \chi}} dt\, e^t \,e^{- e^{t}/ \epsilon}f(e^t) \, .
\end{equation}
Now consider changing variable to $t = \log \chi + \frac{\epsilon s}{\chi}$ and expanding to first order in $\epsilon$ to give
\begin{equation}
    \epsilon e^{-\chi /\epsilon} \left( \int_{\mathcal{H}_0} ds \, e^{-s}f(\chi + s \epsilon)  + O(\epsilon) \right)\,.
\end{equation}
On the other hand, consider the coefficient integral from lemma \ref{lemma:hankelcoefficients} which says
\begin{equation}\label{eq:coefficienteg}
    f_n = \frac{1}{2 \pi i}\int_{\mathcal{H}_{\log \chi}} dw \, e^{-nw}f(e^w)\,,
\end{equation}
One already notes the similarities between \eqref{eq:hankeleg} and \eqref{eq:coefficienteg}---one is a Hankel contour integral on the plane and the other on the cylinder. If we consider changing variables to $w = \log \chi + \tilde{s}/n$ and expanding to leading order in $n$ we find
\begin{equation}
    f_n = \frac{1}{2 \pi i}\frac{\chi^{-n}}{n} \int_{\mathcal{H}_0} d\tilde{s} \, e^{-\tilde{s}} f(\chi + \chi \frac{\tilde{s}}{n} + \ldots ) \,,
\end{equation}
note the extra factor of $\chi$ inside $f$ in the coefficient integral in contrast with the Stokes integral. In summary, we see that if a perturbative series (or Borel power series about $w=0$) has the leading asymptotics:
\begin{equation}
    y_{n+1}(z) \sim \frac{\Gamma(n)}{\chi^n}g(n,z) \quad \iff \left(y_B(z)\right)_n  \sim \frac{g(n,z)}{n \chi^n} \,,
\end{equation}
then the following term enters the trans-series across the Stokes' line $l_{\chi}$ to leading order in $\epsilon$
\begin{equation}
    \left[2\pi \i \epsilon g\left(\frac{\chi}{\ep},z\right) \right] \e^{-\chi / \epsilon}.
\end{equation}
We note that one does not need to understand the nature of the Borel singularity (which may be transcendental and difficult to analyse) in order to understand the more general leading-order Stokes switching. In applications, one may consider fitting $g(n,z)$ to a numerical analysis of the late term behaviour of $y_n(z)$ and immediately deduce the exponentially small contributions.

In section \ref{sec:applications} we discuss how the typical factorial-over-power assumption\footnote{The case $g(n,z) = (n/\chi(z))^{\alpha}$} may be violated in simple examples. We consider more general $g(n,z)$ functions and discuss the implication for singularities in the Borel plane. We now conclude this section with a simple concrete example of a parametric resurgent function.

\begin{example}[Stokes lines from parametric trans-series] \label{example:stokesparam}
Consider the parametric holomorphic function,
\begin{equation}
    y_B(w,z) = \frac{1}{w-z}\frac{1}{w-(z^2-1)} \,.
\end{equation}
This has a singular set in the Borel plane given by $\Gamma_w(z) = \{\chi_1(z) = z\,, \,\chi_2(z) = z^2 -1 \}$. We may recover an asymptotic series from $y_B(w,z)$ via the inverse Borel transform, i.e.
\begin{equation}
    y(z;\epsilon) = \int_{(0,\infty)} \de{w} \,\e^{-w/\epsilon} \, y_B(w,z) = \frac{1}{z(z^2-1)} + \frac{z^2 +z-1}{z^2(z^2-1)^2} \epsilon + \ldots
\end{equation}
and we see the singularities in the physical plane, $\mathbb{C}_z$, are given by $\Gamma_z = \{-1,0,1\}$.
The Stokes graph $l = l_1 \cup l_2$ is given by the locus,
\begin{align}
    l_1 &= \{ z: \Im z^2 = 0, \Re z^2 \ge 0\}, \\
    l_2 &= \{ z: \Im (z^2-1) = 0, \Re (z^2-1) \ge 0\},
\end{align}
which we illustrate in figure \ref{fig:stokesexample}. Then, by computing the local Taylor series expansions around $w=\chi_1(z)$ and $w=\chi_2(z)$, we may deduce non-perturbative corrections around the saddle points to arbitrarily high order in $\epsilon$. For example crossing the Stokes line associated to $\chi_1$ gives the contribution
\begin{equation}
    y_{l_1} = \frac{e^{-z^2 / \epsilon}}{-z^2+z-1} \,.
\end{equation}
\begin{figure}[htb]
    \centering
    \includegraphics{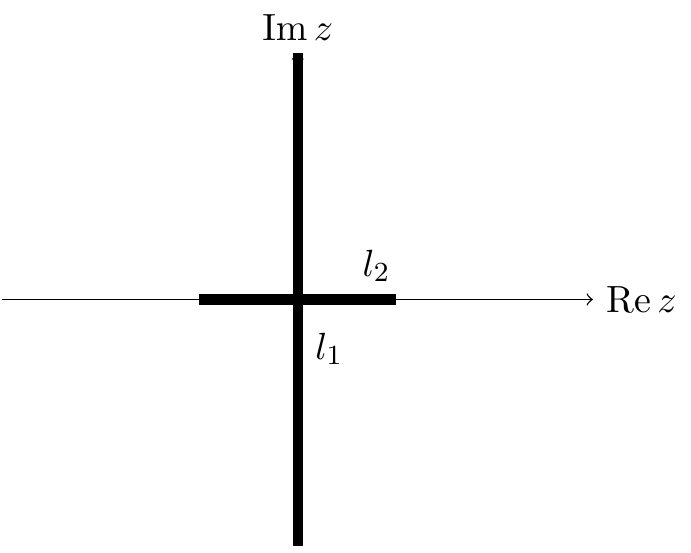}
    \caption{Stokes graph, $l = l_1 \cup l_2$, from \cref{example:stokesparam}.}
    \label{fig:stokesexample}
\end{figure}

\end{example}

\subsection{Boundary layers in parametric trans-series}\label{subsec:boundarylayers}

In contrast to the Borel resummation of constant (non-parametric) trans-series, holomorphic trans-series of the form \eqref{eq:parametricseries} often exhibit the complication of \textit{boundary layers}, or distinguished limits in both $\ep$ and $z$. Informally, when one has a perturbative series with an additional holomorphic parameter, $z$, say $y(z,\epsilon) = y_0(z) + \epsilon y_1(z) + \ldots$, the asymptotic expansion may not only re-order as $n \to \infty$, but may also reorder as $z$ tends to certain critical values. We seek to formalise this notion on the complex-analytic side of the correspondence in \eqref{eq:correspondence}---in brief, we may think of boundary layers as points in $\mathbb{C}_z$ where the radius of convergence of the Borel germ, at $w=0$, shrinks to zero. We begin with a motivating example.

\begin{example}[Boundary layers in the Borel plane]
Consider the parametric function on $\mathbb{C}_w \times \mathbb{C}_z$ given by
\begin{equation} \label{eq:32_yb}
    y_B(w,z) = \frac{1}{w-(z^2-1)}.
\end{equation}
We write the expansion of \eqref{eq:32_yb} around $w=0$, then take the inverse Borel transform to yield the perturbative germ
\begin{equation}\label{eq:boundarylayereg}
    y(z; \,\epsilon) = \epsilon \frac{1}{1-z^2} - \epsilon^2 \frac{1}{(1-z^2)^2} + \ldots.
\end{equation}
From \eqref{eq:32_yb}, note that the Borel singularities are given by
\begin{equation}
    \Gamma_w(z) = \{ w: w = \chi(z) = z^2 -1\} \subset \mathbb{C}_w, 
\end{equation}
and the singularities in the physical plane $\mathbb{C}_z$ of the low order perturbative terms in \eqref{eq:boundarylayereg} are denoted
\begin{equation}
    \Gamma_z = \{ z: z = \pm 1\} \subset \mathbb{C}_z.
\end{equation}
We note that the function $y_B(w,z)$ itself is not singular at the points $\Gamma_z$ for all $w$; rather the pole at $w=\chi(z)$ moves to the origin in $\mathbb{C}_w$ rendering the power series around $w=0$ ill-defined. 

One may approach the point $(w=0,z=z_{\star}) \in \mathbb{C}_w \times \mathbb{C}_z$ along two different directions. First, as above, we may expand the function near $w=0$ and then send $z$ to $z_{\star}$. Alternatively, we may think of $y_B(w,z)$ as a function on $\mathbb{C}_z$ with parametric dependence on $w \in \mathbb{C}_w$ and consider a locally convergent expansion. For example, around $z_{\star}=1$ we have from \eqref{eq:32_yb},
\begin{equation}
    y_B(w,z) = \frac{1}{w} + \frac{2(z-1)}{w^2} + \frac{(4+w)(z-1)^2}{w^3} + \ldots
\end{equation}
We see that at $z_{\star}=1$ we obtain a simple pole at $w=0$. One could not (easily) learn about the nature of this singularity from the expansion \eqref{eq:boundarylayereg}. In \cref{subsubsec:innerouter} we discuss the details of dealing with this phenomena by introducing holomorphic \textit{inner variables}.
\end{example}
Motivated by the above example we define a \textit{boundary layer} of a holomorphic parametric function to be a point $z_{\star} \in \mathbb{C}_z$ where the corresponding $y_B(w,z)$ (thought of as a function on $\mathbb{C}_w$ with parametric dependence on $z \in \mathbb{C}_z$) develops a singularity at $w=0$. We denote this set by $\Gamma_z \subset \mathbb{C}_z$ and throughout the work we assume it is discrete.

We summarise the key ideas from the above discussion. If $y_B(w,z)$ is a parametric resurgent function with a singularity at $w=\chi(z)$ lying on the same sheet as the perturbative germ then, whenever $z_{\star}$ is such that $\chi(z_{\star})=0$, some $y_n(z)$ must be singular at finite $n$.\footnote{Where the singularity appears in the perturbative expansion depends on the nature of the singularity of $y_B(w=0,z)$ at $z=z_{\star}$.} Conversely, suppose that some $y_n(z)$ is singular at a $z=z_{\star}$ then it must be the case that there is a singularity $w=\chi(z)$ on the sheet of $w=0$ that moves to the origin with $\chi(z_{\star})=0$. Hence, boundary layers of the perturbative germ coincide with zeros of $\chi$. In fact, there may be additional  boundary layers $\Gamma^{\chi}_z \subset \mathbb{C}_z$ associated to each $\chi$. The case discussed above, singularities of the perturbative germ $y_n(z)$, defines $\Gamma_z = \Gamma^{0}_z$ but germs around $w=\chi(z) \in \Gamma_w$ may have $\Gamma^{\chi}_z \neq \Gamma_z$---these correspond to singular points of early $a^{\chi}_i(z)$ in the local Borel expansions discussed in the previous section. In this work we focus only on inner-outer matching to the perturbative germ so consider $\Gamma_z$ only.

\begin{remark}
It is an interesting question to understand the complex formalism of how one sets a boundary condition for a differential equation at a boundary layer. It is expected that in this case, one may obtain an additional `constant' resurgent trans-series problems for the coefficients $c_{\chi}$ in the singularly perturbed trans-series:
\begin{equation}
    y(z,\epsilon) = c_{\chi}\sum_{\chi} e^{-\chi(z)/\epsilon}y^{\chi}(z,\epsilon) \,.
\end{equation}
Indeed the work of Howls~\cite{howls2010exponential} finds a similar structure. Here, we mostly avoid this issue by always assuming that the Borel transform, with the perturbative contour $\gamma=(0,\infty)$, satisfies the boundary condition.
\end{remark}

\section{Singular perturbation theory and differential equations} \label{sec:singularperturbation}

The previous two sections dealt with the resurgent properties of asymptotic sequences studied in isolation. We now primarily study sequences generated by singularly perturbed ordinary differential equations. We consider divergent trans-series of the general form
\begin{subequations}
\begin{equation}\label{eq:parametricseries}
    y(z;\epsilon) = y^{0}(z; \epsilon) + \sum_{\chi} e^{-\chi(z)/\epsilon} y^{\chi}(z; \epsilon) \,.
\end{equation}
We refer to the \textit{perturbative germ} or \textit{base series} as 
\begin{equation}\label{eq:baseseries}
y^{0}(z; \ep) \equiv y_0(z) + \epsilon y_1(z) + \epsilon^2 y_2(z) + \ldots + \ep^n y_n(z) + \ldots,
\end{equation}
while we refer to the \textit{fluctuations}\footnote{This terminology is motivated by analogy with theoretical physics. The perturbative evaluation of path integrals organises as a sum over perturbative and non-perturbative sectors. The non-perturbative sectors arise as semi-classical saddle points of the path integral. Perturbative contributions around the saddle points correspond to loop fluctuations in the relevant quantum field.} \textit{around the saddles} ($e^{-\chi(z)/\epsilon}$) as 
\begin{equation} \label{eq:fluctuationseries}
y^{\chi}(z; \ep) \equiv \epsilon^{-\alpha_\chi}\Bigl[y_0^{\chi}(z) + \epsilon y_1^{\chi }+ \ep^2 y_2^{\chi}(z) + \ldots \ep^n y_n^\chi(z) + \ldots \Bigr],
\end{equation}
\end{subequations}
for constant $\alpha_\chi$. The summation in \eqref{eq:parametricseries} is taken over all the $\chi(z)$ associated to singularities, $\Gamma_w$, in the Borel plane. Note that we choose to distinguish the base series, with $\chi=0$, since many of our later examples will involve a solution approximated to leading-order by an algebraic expansion in $\ep$. 

The kind of examples studied in this work are singularly perturbed inhomogeneous ordinary differential equations such as, \textit{e.g.}
\begin{equation}\label{eq:egperturbedequation}
    \epsilon^2 y''(z) + z \epsilon y'(z) + (1-z) y (z) = \frac{1}{z} \,,
\end{equation}
where $y(z)$ decays to zero in appropriate sector(s) of the complex plane. The terminology \textit{singularly perturbed} arises since the solutions to such equations with $\ep \to 0$ are qualitatively different to solutions at $\epsilon=0$. Naive asymptotic expansions to such equations yield perturbative series of the form $y(z; \ep) = y_0(z) + \epsilon y_1(z) + \epsilon^2 y_2(z) + \ldots$ and hence we are concerned with the parametric expansions given by \eqref{eq:parametricseries} rather than the constant-coefficient expansions of the type \eqref{eq:yser}. 

In this section, our goal is to develop the right hand-side of the correspondence \eqref{eq:correspondence} and discuss how problems in singular perturbation theory can be studied via the parametric complex function, $y_B(w;z)$, on the Borel plane. The basis of this approach relies upon the behaviour of $y_B$ near power law singularities, $w = \chi$, where we may write locally
\begin{equation} \label{eq:yB_chiexpand}
    y_B(w,z) = (w-\chi(z))^{-\alpha_{\chi}}(a_0^{\chi}(z) + (w-\chi(z))a_1^{\chi}(z) + \ldots) + \text{reg.}
\end{equation}
Then, given the above knowledge of the local expansion of $y_B(w,z)$, one may apply the inverse Borel transform \eqref{eq:parametricborel_inversion} and determine (\cref{lemma:hankel}) the relationship between the components $\{\alpha_\chi, a_i^{\chi}(z)\}$ and the fluctuations, $y^{\chi}(z)$, in the trans-series expansion \eqref{eq:parametricseries}.

\paragraph{Parametric replacement rules.}
Differential equations of the general form \eqref{eq:egperturbedequation} may be transferred to operators on functions on Borel space using the following straightforward lemma:
\begin{lemma}
In Borel space the operators $d/dz$ and multiplication by $\epsilon$ become
\begin{equation}
    \mathcal{B}_z: \frac{d}{dz} \to \partial_z \quad \text{and} \quad
    \mathcal{B}_z: \epsilon^{-1} \to \partial_w \,.
\end{equation}
\end{lemma}
\begin{proof}
This is a consequence of the definition of the inverse Borel transform in \eqref{eq:parametricborel_inversion}.
\end{proof}

In this work, we shall primarily focus on the case of singularly perturbed linear inhomogeneous $N$th-order differential equations for $y = y(z;\epsilon)$ given by $\mathscr{P}y = F(z)$. As a result of the above transformation rules, the operator is mapped as follows:
\begin{equation}\label{eq:generalborelPDE}
    \mathscr{P} = \sum_{i=0}^N \epsilon^{i} P_i(z) \frac{d^i}{dz^i} \to \mathscr{P}_B = \sum_{i=0}^N  P_i(z) \partial_z^i \partial_{w}^{N-i} \,.
\end{equation}
\begin{example}[First-order example of $\mathscr{P}$ and $\mathscr{P}_B$] Using the above lemma, we may write the correspondence between the following first-order operator and its Borel analogue:
\begin{equation}
\mathscr{P} = \epsilon \dd{}{z} + G(z) \to 
\mathscr{P}_B = \partial_z + G(z) \partial_w.
\end{equation}
If we solve $\mathscr{P} y = F(z) = \epsilon H(z)$ with an inhomogeneous term, then it follows from the definition \ref{eq:parametricboreltransform} that, in the Borel plane, we are led to solve $\mathscr{P}_B y_B = 0$ together with the `initial data' $y_B(w=0,z) = H(z)/G(z)$. We discuss the straightforward general solution to such equations in more detail in \cref{subsec:firstorder}.
\end{example}

Note that in contrast to \cref{sec:background}, we now study the Borel transform, $y_B$, given by a \textit{partial} differential equation, $\mathscr{P}_By_B = 0$. In the context of the exact WKB analysis of Schr\"odinger equations \cite{voros1983return} the resultant partial differential equations have been studied in great detail. For example, the work of Takei, Aoki, Iwaki \textit{et. al.} \cite{kawai2005algebraic} argue that the singular locations, say $w = \chi$, may be determined by the study of propagation of singularities of \eqref{eq:generalborelPDE}. Further as demonstrated in \cite{aoki2003exact, aoki2008virtual}, the so-called \emph{microlocal property} of such operators explains the type and location of singularities.

\begin{remark}We highlight that the function, $\chi(z)$, that appears throughout this work is referred to as a \textit{singulant}, primarily when it is used in conjunction with exponential switchings, $\e^{-\chi/\ep}$, or alternatively in the context of the late terms, \eqref{eq:latetermsansatz}. The terminology is due to Dingle \cite{dingle1973asymptotic}.
\end{remark}

\subsection{A singularity ansatz in the Borel plane}\label{subsec:ansatz}

In our work, we shall take a slightly different approach and explain the Borel-plane analogue of the factorial-over-power late terms ansatz approaches of applied exponential asymptotics (\textit{e.g.} \cite{chapman_1998_exponential_asymptotics}): we shall posit a singularity ansatz for $y_B$, to be satisfied via its governing partial differential equation. In this case, one can show that the bi-characteristic curves studied in the exact WKB approach yield effective ordinary differential satisfied by the singulant $\chi(z)$. 

Given the base series \eqref{eq:baseseries}, the factorial-over-power approach of applied exponential asymptotics posits the following ansatz for the large $n$ asymptotics
\begin{equation} \label{eq:latetermsansatz_pre}
    y_{n+1}(z) = \frac{n^{\alpha_{\chi}-1} n!}{\chi(z)^{n+\alpha_{\chi}}}\left(b_0^{\chi}(z) + \frac{1}{n} b_1^{\chi}(z)+ \ldots \right) \qquad \text{as $n \to \infty$}.
\end{equation}
It is also common to write the above in the equivalent form, 
\begin{equation}\label{eq:latetermsansatz}
    y_{n+1}(z) = \frac{\Gamma(n+\alpha_{\chi})}{\chi(z)^{n+\alpha_{\chi}}}\left(a_0^{\chi}(z) + \frac{1}{n+\alpha_{\chi}-1} a_1^{\chi}(z) +\ldots \right) \qquad \text{as $n \to \infty$}.
\end{equation}
The relationship between the two above expansions was previously discussed near equations \eqref{eq:ffactpow_1} and \eqref{eq:ffactpow}---there we demonstrated that the same conclusions about the exponential trans-series components can be drawn from either expression. Instead of the above, we may study the problem in the Borel plane, where we use assumption that $y_B$ exhibits a power law singularity at $w = \chi$. Hence we posit the ansatz
\begin{equation}\label{eq:ansatz}
    y_B(w,z) = (w-\chi(z))^{-\alpha_\chi}(a_0^\chi(z) + a_1^\chi(z)(w-\chi(z)) + \ldots) \,,
\end{equation}
and then seek to solve for the components, $\chi$, $\alpha_\chi$, and $a_i^\chi$ and consequently derive the trans-series \eqref{eq:parametricseries}.\footnote{Note the functions $a_i^\chi$ in \eqref{eq:latetermsansatz} and \eqref{eq:ansatz} differ by numerical factors, but for ease of presentation, we shall sometimes denote them as the same. Lemmas \ref{lemma:hankel} and \ref{lemma:hankelcoefficients} present the correspondences between the two.} We demonstrate the Borel-plane ansatz by example.

\begin{example}[Borel-plane ansatz for a second-order ODE] 
\label{example_2ndorderkey}
Key illustrative examples in this work will be singularly perturbed inhomogeneous second-order linear ODEs of the form
\begin{equation} \label{eq:secondorderODE}
    \mathscr{P}y = \epsilon^2 y''(z) + \epsilon P(z) y'(z) + Q(z) y(z) = F(z) \,,
\end{equation}
with $P(z)$, $Q(z)$ and $F(z)$ meromorphic functions on $\mathbb{C}_z$. For concreteness, we assume that the boundary conditions for $y$ are specified by the inverse Borel transform with $\gamma_0 = (0,\infty)$. However, note that the boundary conditions are not relevant to the following analysis in determining the Stokes lines and associated Stokes jumps. Expanding $y = y_0 + \ep y_1 + \ldots$, we obtain from \eqref{eq:secondorderODE} the first two orders,
\begin{equation} \label{eq:2ndorder:y0y1}
    y_0(z) = \frac{F(z)}{Q(z)} \quad \text{and} \quad y_1(z) = - \frac{P(z)}{Q(z)}y_0'(z) \,,
\end{equation}
whose singular boundary layer set, $\Gamma_z$, on $\mathbb{C}_z$ consists of poles of $F(z)$ and $P(z)$, and zeroes of $Q(z)$. In Borel space, $\mathbb{C}_w$, the relevant operator is
\begin{equation} \label{eq:PB_2ndorder}
    \mathscr{P}_B = \partial^2_z + P(z)\partial_z \partial_w + Q(z) \partial^2_w \,.
\end{equation}
Now we consider a singularity ansatz for $y_B$ of the form of \eqref{eq:ansatz}. 
Seeking a solution to $\mathscr{P}_B y_B=0$ at leading order with $w$ near $\chi(z)$, we obtain an algebraic equation for $\chi'(z)$:
\begin{equation} \label{eq:chipeqn_example}
    \left(\chi'\right)^2 - P(z) \chi' + Q(z) = 0 \,,
\end{equation} 
and hence two first-order differential equations for the singulant, $\chi(z)$. Note that the multi-valued nature of $\chi(z)$ may be viewed as a result of treating $\mathbb{C}_w$ and $\mathbb{C}_z$ differently---this is in contrast to the singularity propagation approach popular in exact WKB analysis.

To proceed, it is convenient to make the affine transformation $t = w - \chi(z)$ to move the singularity to the origin. In this new variable, we obtain a new effective Borel PDE centred on $\chi(z)$ given by
\begin{equation}\label{eq:shiftedborel}
    \mathscr{P}_B = (P-2\chi'(z))\partial_z \partial_t - \chi''(z) \partial_t + \partial^2_z \,.
\end{equation}
Using \eqref{eq:chipeqn_example} and comparing orders in $t$ of $\mathscr{P}_B y_B = 0$ yields a set of recurrence relations for the coefficients, $a_n^\chi(z) = a_n(z)$, with $n=0, 1, 2, \ldots$ given by
\begin{subequations} \label{eq:a0an_gen}
\begin{gather}
    (P(z)-2\chi'(z))a_0'(z) - \chi''(z) a_0(z) = 0 \,,  \label{eq:a0_ode} \\
    (n-\alpha)a'_n(z) (P(z)-2 \chi'(z)) + a''_{n-1}(z) -\chi''(z)(n-\alpha) a_n(z) = 0 \,,  \quad \text{for $n > 0$}. \label{eq:secondorderrecursive}
\end{gather}
\end{subequations}
Note that there is an apparent absence of initial conditions for the differential equations satisfied by the $a_n(z)$, i.e. \eqref{eq:a0_ode} and  \eqref{eq:secondorderrecursive}. This is an important observation and, in the following section, we will find initial conditions at singular points $z_\star\in\Gamma_z$ using an inner-outer matching procedure in the complex plane.

As a further example, note that if $P$ and $Q$ in \eqref{eq:a0_ode} are now holomorphic on $\mathbb{C}$ then by the linearity of the differential equation, $a_0$ has possible singular points only at points $z \in \mathbb{C}_z$ where $P(z) = 2 \chi'(z)$. These singularities will, in general, be distinct\footnote{Indeed, as previously remarked, the singularities of low-order $a_i(z)$ determine a new set of physical boundary layers $\Gamma^{\chi}_z$ for the local expansion at $w=\chi(z)$.}  from those in $\Gamma_z$. 
\end{example}

\begin{example} \label{eg:appendixeg}
Consider \eqref{eq:secondorderODE} with $P=2$ and $Q = 1-z$ so that the singularity locations satisfy $\chi_{\pm}' = 1 \pm z^{1/2}$ from \eqref{eq:chipeqn_example}.
Assuming $\alpha$ is not an integer (which depends on the choice of $F$), we find from \eqref{eq:secondorderrecursive}, in the case of $\chi = \chi_+$: 
\begin{equation} \label{eq:2ndorder_a0a1}
    a_0(z) = a_0^{\chi_+}(z) = \frac{c_0}{z^{1/4 }} \quad \text{and} \quad
    a_1(z) = a_1^{\chi_+}(z) = \frac{c_1}{z^{1/4}} + \frac{5c_0}{48(\alpha-1)z^{7/4}},
\end{equation}
for (Stokes) constants $c_0$, $c_1$, and so forth. In accordance with the discussion in \cref{example_2ndorderkey}, the coefficients in \eqref{eq:2ndorder_a0a1} are singular at the origin in $\mathbb{C}_z$ only. The singularities define a new set (in this case one) of boundary layers in the physical plane $\Gamma^{\chi}_z=\{0\} \subset \mathbb{C}_z$ associated to $\chi_+ \in \Gamma_w(z)$.
\end{example}

\begin{example}
Let us consider again the general second order singularly perturbed linear ODE
\begin{equation}
    \mathscr{P}y = \epsilon^2 y'' + \epsilon P(z)y' + Q(z) y \,,
\end{equation}
with associated linear Borel operator
\begin{equation}
    \mathscr{P}_B y_B = \partial_z^2 + P(z)\partial_z \partial_w + Q(z)\partial_w^2 \,.
\end{equation}
We now explain precisely how the singularity ansatz (or equivalently the factorial-over-power ansatz) method for obtaining the singulant $\chi(z)$ is related to the exact WKB bicharacteristic method. It is argued in \cite{aoki1993microlocal} that singularities of $\mathscr{P}_B$ propagate along bicharacteristics in $T^* (\mathbb{C}_z \times \mathbb{C}_w)$. These are curves defined by Hamiltonian flow (with respect to the natural complex symplectic structure on the cotangent space) defined by the symbol $\sigma$ of $\mathscr{P}_B$. Introducing coordinates $(z,w)$ for $\mathbb{C}_z\times \mathbb{C}_w$ and $(\xi,\eta)$ for the respective cotangent directions then we see that in the second order example the symbol is given by
\begin{equation}
    \sigma = \xi^2 + P(z)\xi \eta + Q(z) \eta^2 \,,
\end{equation}
and the associated Hamilton equations are given by
\begin{align}
    &\dot{z} = \partial_{\xi}\sigma = 2 \xi +P(z) \eta\,, \quad &\dot{w} = \partial_\eta \sigma = P(z)\xi + 2 Q(z) \eta \,, \\
    &\dot{\xi} = - \partial_z \sigma = P'(z)\xi \eta + Q'(z) \eta^2 \,,\quad &\dot{\eta} = - \partial_w \sigma = 0 \,,
\end{align}    
together with the `energy' condition $\sigma=0$. In the above equations the dot superscript denotes differentiation with respect to a `time' parameter $t \in \mathbb{C}$ so that solutions are parametric curves in $T^*( \mathbb{C}_z \times \mathbb{C}_w)$. If instead we seek the projection of solutions to $\mathbb{C}_z \times \mathbb{C}_w$ in terms of $w = \chi(z)$ then we find
\begin{equation}
    \chi'(z) = \frac{dw}{dz} = \frac{\dot{w}}{\dot{z}} = \frac{P(z)\xi + 2 Q(z)}{2 \xi + P(z)} \,,
\end{equation}
where we have used the fact that we are free to set $\eta = 1$. Using the relation $\sigma=0$ it is straightforward to verify
\begin{equation}
    \chi'(z)^2 - P(z) \chi'(z) + Q(z) = 0\,,
\end{equation}
and we recover the singulant equation \eqref{eq:chipeqn_example}.
\end{example}

\subsubsection{Discussion}
To review, the coefficients $a_i^\chi(z)$ which appear in the singularity ansatz of $y_B$ in \eqref{eq:ansatz}, may be viewed from three perspectives.
\begin{enumerate}
\item By definition, they represent the series coefficients when the Borel transform, $y_B$, is expanded about singularities, $w = \chi(z)$, in $\Gamma_w$. 
\item The $a_i^{\chi}(z)$ appear in the $1/n$ expansion of the base series \eqref{eq:baseseries} coefficients $y_n(z)$. See \eqref{eq:latetermsansatz}.
\item By Borel inversion and the Hankel contour formula~\eqref{cor:stokesswitching}, the coefficients $a_i^\chi$ are directly related to the trans-series for the saddle fluctuations, $y^\chi$, in \eqref{eq:fluctuationseries}. 
\end{enumerate}
These are straightforward consequences of the definition of the Borel transform and the argument of section \ref{sec:background} that relates the $1/n$ trans-series of the late terms to the corresponding $\epsilon$ trans-series.
Note that this structure persists to `higher singularities' in the sense that the same conclusions may be drawn from the re-scaled expansion around a singularity $\chi(z)$ (that is, consider the $a_i^{\chi}(z)$ as defining a new Borel germ) and the next nearest singularity $\tilde{\chi}(z)$. That is to say, the late terms of $y^{\chi}(z)$ may again be factorially divergent. We conclude with some remarks.

\begin{remark}[Connection of $y_B$ singularity ansatz and trans-series fluctuations $y^\chi$]
According to the discussion at the beginning of \cref{subsec:ansatz}, recall that the recurrence relation \eqref{eq:secondorderrecursive} may be viewed in two main ways: either as the recurrence relation of the $1/n$ expansion of the late terms of $y_n(z)$, or as the differential-recurrence relation satisfied by the higher saddle perturbative series $y^{\chi}(z,\epsilon)$. Indeed, if one substitutes 
\begin{equation}
    y_B(t + \chi(z), z) = t^{-\alpha_{\chi}}y_B^{\chi}(t,z) \,,
\end{equation}
into the Borel PDE in the form \eqref{eq:shiftedborel} then the effective equation obtained for $y_B^{\chi}$ is the Borel transform of the effective equation for $y^{\chi}(z,\epsilon)$, obtained by making a trans-series ansatz of the form \eqref{eq:parametricseries}.
\end{remark}

\begin{remark}
Earlier, we discussed some of the connections between the work here and the exact WKB analysis of \textit{e.g.}~\cite{aoki2003exact,kawai2005algebraic}. For linear problems the bicharacteristic flow coincides with the equation $w=\chi(z)$ for the location of the singularity. However the ansatz we present (which parallels the factorial-over-power late terms ansatz), though non-rigorous,\footnote{For example it is unclear that our ansatz really has the required error bounds in general.} allows for a more general determination of the curves in $\mathbb{C}_z \times \mathbb{C}_w$ along which $w= \chi(z)$ propagates without relying, necessarily, on a linear $\mathscr{P}_B$ operator. In particular, for non-linear problems, whilst different sectors of the trans-series \eqref{eq:fn_trans} may be coupled, it is still possible to obtain an effective equation for $\chi(z)$---indeed this is the spirit of the nonlinear ODE analyses of exponential asymptotics \cite{chapman_1998_exponential_asymptotics}. 
\end{remark}

In the remainder of this section we will explain precisely how the typical methodology to determine the functions in the factorial-over-power ansatz may be mapped to the Borel plane. The main idea is as follows. If the perturbative series, $y^{0}(z)$ or $y^\chi(z)$ in \eqref{eq:parametricsequence} arise as solutions to a singularly perturbed differential equation, then many of the features of the trans-series are fixed by the operator form of $\mathscr{P}$ (or $\mathscr{P}_B$). As we have seen via \cref{example_2ndorderkey}, $\mathscr{P}_B$ imposes certain ordinary differential equations on the components $\chi(z)$ and the associated $a_i^{\chi}(z)$ in \eqref{eq:yB_chiexpand}. The initial conditions of these ODEs are not known \textit{a priori}. Thus, we may interpret $\mathscr{P}$ or $\mathscr{P}_B$ as yielding `blank template' trans-series that must now be populated with constants or initial data. The late-terms ansatz methods of \textit{e.g.}~Chapman \textit{et al.}~\cite{chapman_1998_exponential_asymptotics}, Kruskal and Segur~\cite{kruskal1991asymptotics} explains how these (essentially Stokes) constants may be fixed using methods of matched asymptotics. 

One of the primary goals of the present work is to explain the above procedure via  the complex-analytic side of the correspondence in \eqref{eq:correspondence}. 
In short, we demonstrate that the inner equation is a $z$-independent connection problem of the type studied in section \ref{sec:background}---the trans-series solution of which determines the required initial data at boundary layers $\Gamma_z$ for the $z$-dependent problem.

\subsection{Matching to the origin in the Borel plane} \label{sec:matchingzero}

The next part of the ansatz procedure is to identify the initial data involved in determining $\chi(z)$ (see \emph{e.g.}~\eqref{eq:chipeqn_example}) and also the constant $\alpha$, which appears in \eqref{eq:ansatz} and determines the nature of the singularity at $w = \chi(z)$. In the applied exponential asymptotics 
literature, these steps are often associated with ensuring that the ansatz \eqref{eq:ansatz} is `consistent with the low-order perturbative expansion'.\footnote{In the general case, this is only a heuristic that one may check in a variety of examples, perhaps \textit{a prosteriori} with numerical analysis of the problem. In particular, e.g. for non-linear problems, it is not at all clear that the solution found by the ansatz is unique. Of course our approach suffers from the same difficulties but phrasing the problem in the Borel plane at least leads the way to more rigorous approaches.}

Let us suppose that $w=\chi(z)$ lies on the same branch as our perturbative germ at $w=0$. From the general arguments of section \ref{subsec:boundarylayers} we know that $\chi(z)$ must be zero at singularities of the low-order terms $y_1(z),y_2(z),\ldots$, this fixes initial conditions for $\chi(z)$. Namely, let us consider a singulant, $\chi$, with
\begin{equation}
    \chi(z_{\star}) = 0  \quad \text{at} \quad z_{\star} \in \Gamma_z \,.
\end{equation}

We now work in a neighbourhood of a point $z_{\star} \in \Gamma_z$ so that $\chi(z)$ becomes the nearest singularity to $w=0$. The setup for this subsection is illustrated in figure \ref{fig:originmatch}. 

\begin{figure}[htb]
    \centering
    \includegraphics{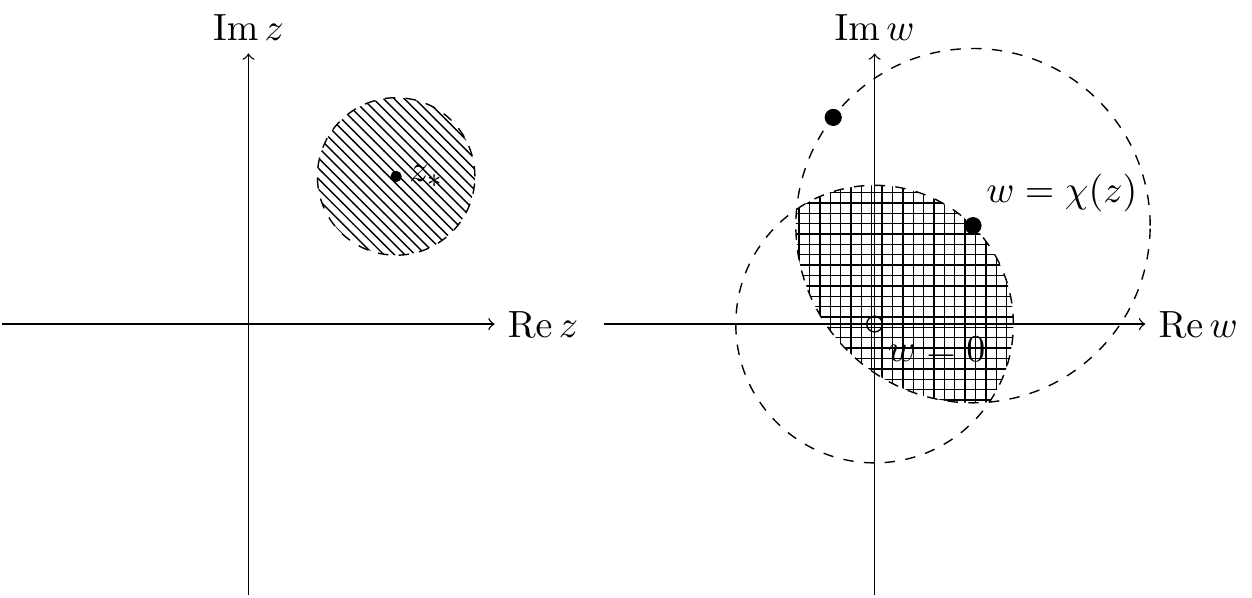}
    \caption{(Left) A singularity, $z^*\in\Gamma_z$, of the early perturbative terms $\{ y_0(z),y_1(z), \ldots \}$. The region of convergence is shaded. (Right) Comparing expansion near $w=\chi(z)$ with the expansion near $w=0$. The shaded region corresponds to where matching can occur.}
    \label{fig:originmatch}
\end{figure}

Recall that the local expansion of $y_B(w, z)$ with a power law singularity, $w = \chi(z)$, takes the form \eqref{eq:ansatz}. In accordance with the arguments of section \ref{subsec:boundarylayers} we may consider $y_B$ at $w=0$ and compare with the perturbative solution near the origin
\begin{equation}\label{eq:loworders}
    y_B(w=0,z) = (-\chi(z))^{-\alpha}a_0(z) + \ldots = y_1(z) \,,
\end{equation}
the power law singularity ansatz in the Borel plane is then consistent if we may expand locally about the point $z_{\star} \in \Gamma_z$ to determine $\alpha$. This is a basic assumption of the applied exponential asymptotics literature where the above expansion is interpreted as ensuring that the late-orders ansatz in \eqref{eq:latetermsansatz} at large $n \to \infty$ is consistent with early asymptotic orders. In particular, setting $n=0$ in the late terms ansatz $\eqref{eq:latetermsansatz}$ we find 
\begin{equation}
    y_1(z) = (-\chi(z))^{-\alpha}a_0(z) + \ldots
\end{equation}
which clearly gives an equivalent consistency condition to \eqref{eq:loworders}. However in the present work this constraint is interpreted as a patching condition for local holomorphic expansions in the Borel plane.

\begin{example}\label{ex:localzero}
Suppose $\chi(z)$ has a zero of order $\gamma$ at a point $z_{\star} \in \Gamma_z$. Let us then expand
\begin{subequations} \label{eq:local_chi_y1_ai}
\begin{align}
    \chi(z) &= X_1 (z - z_\star)^\gamma + \ldots \\
    y_1(z) &= d_1 (z - z_\star)^{-\delta} + \ldots \\
    a_{i}(z) &= a_i(z-z_{\star})^{\beta_i} + \ldots 
\end{align}
\end{subequations}
for the singulant, leading perturbative term, and saddle coefficient in \eqref{eq:ansatz}, respectively. Above, $X_1$, $d_1$, and $a_i$ are non-zero constants. Note, for example, from the discussion on second-order differential equations above that when $P$ is constant, then $a_i(z)$ for $i=1,2,\ldots$ are non-singular at roots of $\chi(z)$ and so $\beta_i=0$.\footnote{In general one needs a singular differential operator acting here to ensure that $\beta$ does not increase with $i$---the procedure is not necessarily consistent for an arbitrary trans-series.} Now, comparing the singularity expansion of $y_B$ to the perturbative germ, we find
\begin{equation}
    y_B(0,z) = c(z-z_\star)^{-\alpha \gamma + \beta} + \ldots = d_1 (z-z_\star)^{-\delta} +\ldots \,,
\end{equation}
and hence we find that $\alpha = (\beta + \delta)/\gamma$. In general, the constant prefactor, $c$, may acquire infinitely many contributions from the $a_{i}$ constants in the saddle fluctuations, and so these remain undetermined.
\end{example}

\begin{example}\label{eg:forcedappendixeg}
Let us supplement example~\ref{eg:appendixeg} with an inhomogeneous term
\begin{equation}
    F(z) = \frac{1}{(1-z)(2-z)} \,,
\end{equation}
and seek a solution to $\mathscr{P} y = F$ or equivalently $\mathscr{P}_B y_B= 0$ with the initial data $y_B(w=0,z) = y_1(z)$ in the Borel plane. We have
\begin{equation}
    y_1(z) = \frac{2(-5+3z)}{(-2+z)^2(-1+z)^4} \,.
\end{equation}    
We focus on the singularity $z_\star = 1$ and the corresponding singulant, $\chi_+=z+z^{3/2}-2$. From \eqref{eq:local_chi_y1_ai} and \eqref{eq:2ndorder_a0a1}, we find $\gamma = 1$, $\delta = 4$, and $\beta = 0$. Hence via example~\ref{ex:localzero}, we thus find $\alpha= (\beta + \delta)/\gamma = 4$, and therefore an order $4$ pole at $w = \chi_+$ in the Borel plane. Although it was the case here, note that the nature of physical-plane and Borel-plane singularities need not be correlated in general.
\end{example}

\begin{example}
We may generalise the singulant and matching to $N$th-order inhomogeneous linear differential equation, $\mathscr{P}y = F(z)$, with the associated Borel operator given by \eqref{eq:generalborelPDE}. Now, instead of \eqref{eq:chipeqn_example}, $\chi'$ satisfies the algebraic equation
\begin{equation}
    \sum_{i=0}^N P_i(z)(\chi')^i = 0 \,.
\end{equation}
The forcing term, $F(z)$, plays a key role in determining the initial perturbative terms, $y_0(z), y_1(z), y_2(z), \ldots$. We see that there are two distinct types of singularities for forced singularly perturbed differential equations. Firstly, those associated to poles $\{z_1,\ldots,z_f\}$ of $F(z)$ (which we assume throughout is meromorphic)--- in that case we obtain multiple singularities satisfying $\chi(z_i) = 0$ for $i=1,\ldots,f$. And secondly, there are those singularities associated to zeros of $P_i(z)$, denoted $\{z_1,\ldots,z_p\}$. In the latter case, $\chi'(z)=0$ at these points and they may be interpreted as points where $\chi^{-1}(z)$ is multivalued. We shall see both these two types of singularities in the first-order example of section~\ref{sec:applications}.
\end{example}

One intriguing complication involves the subtle assumption at the start of this section~\ref{sec:matchingzero}, which was that generic singularity at $w = \chi(z)$ is assumed to lie on the same branch as the perturbative germ at $w =0$. Indeed, there are situations where $\chi(z)$ may not lie on the initial Riemann sheet, and the above argument of directly matching to $w = 0$ is then invalid. The equation studied in examples \ref{eg:appendixeg} and \ref{eg:forcedappendixeg} is studied in the Appendix of the work of Trinh \& Chapman~ \cite{trinh2013new}; there, it is suggested that the equation exhibits `higher-order Stokes phenomena'. In the present context this implies that the series formed from the $a_i(z)$ has a finite radius of convergence---or equivalently the trans-series contribution $y^{\chi}(z;\epsilon)$ has a factorial divergence in its perturbative expansion not already accounted for by singularities on the same sheet. Loosely speaking, this new singularity $\tilde{\chi}$ lives on a higher sheet above $w=0$.

\begin{figure}
    \centering
    \includegraphics[scale=0.15]{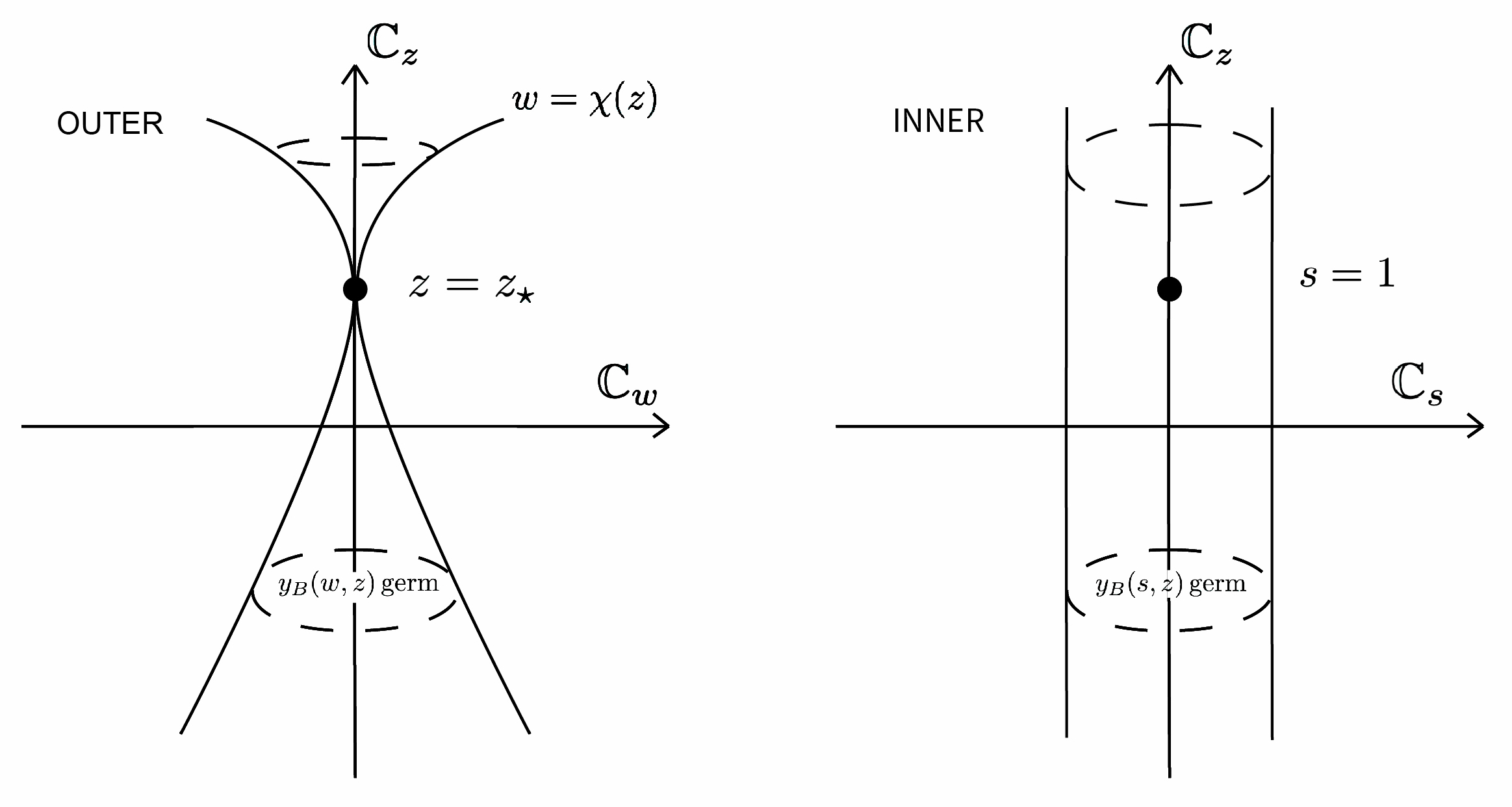}
    \caption{The Borel plane $\times$ the physical plane, shown in the outer (left, $\mathbb{C}_w \times \mathbb{C}_z$) and inner (right, $\mathbb{C}_s \times \mathbb{C}_z$) Borel variables.}
    \label{fig:innerouter}
\end{figure}

\subsection{Inner-outer matching of trans-series}\label{subsubsec:innerouter}
In section~\ref{sec:matchingzero}, we explained a method to determine the location, $w = \chi(z)$, and nature, $\alpha$, of singularities of $y_B(w,z)$. Further, when $y_B(w,z)$ is expanded about $w = \chi(z)$ and the ansatz \eqref{eq:ansatz} is substituted into the Borel PDE, $\mathscr{P}y_B = 0$, we obtain a set of differential equations for the power series coefficients, $a_i^{\chi}(z)$. In this sense we now have a `blank template' trans-series
\begin{equation} \label{eq:transeries_again}
    y(z;\epsilon) = y^{(0)}(z) + \sum_{\chi} e^{-\chi(z)/\epsilon} y^{\chi}(z;\epsilon) \,.
\end{equation}
with the $z$-dependent coefficients in the perturbative expansions of the $y^{\chi}(z,\epsilon)$ essentially coinciding with the $a^{\chi}(z)$ functions in the expansion of $y_B$ about $w = \chi$ (cf. \cref{lemma:hankelcoefficients}).

It turns out that for singularly perturbed differential equations, we may reduce the problem of obtaining initial data for the trans-series components to a `constant' resurgent problem of the type discussed in section \ref{sec:background}. In this section, we focus more on the $\mathbb{C}_z$-plane, which up to this point, has largely played an auxiliary role. Recall that we have a discrete set of distinguished points (with associated boundary layers) $z_{\star} \in \Gamma_z \subset \mathbb{C}_z$ that correspond to singularities of the early terms in the perturbative germ, $y^{(0)}(z)$, in \eqref{eq:transeries_again}---this is where we will set initial data. Expressing $y_B(w,z)$ in a new set of variables that move singularities in $\mathbb{C}_w$ to the unit disk (so that in particular $z$-dependence is lost) reveals a constant resurgent problem of the type discussed in section \ref{sec:background}. We therefore extend the correspondence \eqref{eq:correspondence} to include the method of matched asymptotics on the left hand-side and a version solely in the Borel-plane on the right hand-side.

We note again that it is possible that expansions around $\Gamma_w$ have additional singularities $\tilde{\chi} \in \Gamma^{\chi}_w$ and additional boundary layers $\Gamma^\chi_z$ associated to singularities of low orders of $a_i^{\chi}(z)$ in the same way as above. In this case one must match between $\Gamma^{\chi}_w$ and $\Gamma^{\chi}_z$---and so on \textit{ad infinitum}. The idea is the same but we focus, for notational simplicity, on the `first level' of inner-outer matching in this work and further consider in detail only the sheet of the Borel plane connected to the perturbative germ.

\paragraph{Setup.}
Suppose $y_B(w,z)$ is a parametric resurgent function with singularities at $w= \chi_1(z),\chi_2(z),\ldots$. Pick a particular singularity, $w=\chi(z)$, and let $z_{\star}$ be an element of $\Gamma_z$ with $\chi(z_{\star})=0$. We define a new inner variable
\begin{equation} \label{eq:inners}
    s = \frac{w}{\chi(z)} \,,
\end{equation}
so that $y_B(s,z)$ is singular at the points $s=1$ and also $s = \chi_i(z)/\chi(z)$ for $i\geq 1$. 

By the discussion in section \ref{subsec:boundarylayers}, whenever $z=z_{\star}$ is such that $\chi(z)=0$, then we note the early terms, $y_i(z)$, are singular\footnote{This assumes that $w=\chi(z)$ must lie on the same Riemann sheet as the origin, $w = 0$.} for some $i$ and the asymptotic expansion in $y(z;\epsilon)$ breaks down due to the presence of a boundary layer. The inner variable \eqref{eq:inners} satisfies $s \to \infty$ as $z \to z_\star$. We may now consider expanding our Borel ansatz alternatively as
\begin{equation} \label{eq:innerexpansion_pre}
    y_B(s,z) = (z-z_{\star})^{-\beta}(\varphi_0(s) + (z-z_{\star})\varphi_1(s) + \ldots) \,.
\end{equation}
Without loss of generality, we henceforth assume $z_{\star} = 0$ and write the above as
\begin{equation}\label{eq:innerexpansion}
    y_B(s,z) = z^{-\beta}(\varphi_0(s) + z \varphi_1(s) + \ldots ).
\end{equation}
Note that we may obtain values of $\varphi_i(0),\varphi_i'(0),\varphi_i''(0),\ldots$ by comparing with the low-order perturbative terms. That is, from the expansion of $y_B$ about the origin, we may write
\begin{multline}
    y_0(z) + \bigl[\chi(z) s\bigr] y_1(z) + \bigl[\chi(z) s\bigr]^2 (2!) y_2(z) + \ldots \\
    = z^{-\beta}\Bigl[(\varphi_0(0) + z \varphi_1(0)+z^2 \varphi_2(0) \ldots) 
    + s (\varphi_0'(0) + z \varphi_1'(0)+\ldots) + \ldots\Bigr] \,,
\end{multline}
then expand the early perturbative terms on the left hand-side about the relevant point in $\Gamma_z$ (here $z = 0$). The setup is illustrated in \ref{fig:innerouter} and we now discuss the complex-analytic matching with a simple toy example.

\begin{example}
Consider the function on $\mathbb{C}_w \times \mathbb{C}_z$ given by
\begin{equation} \label{eq:example39_yB}
    y_B(w,z) = \frac{1}{w-\chi_1(z)} + \frac{1}{w-\chi_2(z)} \,,
\end{equation}
where $y_B(w,z)$ has simple poles at $\chi_1(z) = z$ and $\chi_2(z) = z^2 + 1$. We remind the reader that typically the goal is, in some sense, to reconstruct an unknown $y_B(w,z)$ from local data around $w=0$; here we have knowledge of the exact $y_B(w,z)$.  The perturbative data for this problem may be obtained by expanding the Borel transform around $w=0$ 
\begin{equation}
    y_B(w,z) = y_0(z) + w y_1(z) + \ldots 
\end{equation}
from which we may read off the asymptotic series
\begin{equation} \label{eq:example39_perb}
    y(z;\epsilon) = -\left(\frac{1}{z}+\frac{1}{z^2+1}\right) + \epsilon \left(\frac{1}{z^2}+\frac{1}{(z^2+1)^2}\right) + \ldots 
\end{equation}
From the above, we have singular points  $z_{\star} \in \Gamma_z = \{0, \pm i\}$. Crucially, notice that, if we were to examine the analytic expression \eqref{eq:example39_yB}, those points $z_\star$ are largely unremarkable---$w=0$ is simply a poor choice of expansion point near $z_\star$. Let us now define new \textit{inner variables}, 
\begin{equation}
    s_1 = \frac{w}{\chi_1(z)} \quad \text{and} \quad s_2 = \frac{w}{\chi_2(z)} \,.
\end{equation}
These variables send zeros of $\chi_i(z)$ to $s_1, s_2 \to \infty$. Clearly $y_B$ is singular at the points $s_1 = 1$ and $s_2 = 1$. Further, from the discussion in section \ref{subsec:boundarylayers}, we know that when either $s_1 = 0$ or $s_2 = 0$ then $y_B(w,z)$ will be singular at zeros of $\chi_1$ or $\chi_2$. We now compare two dual expansions, thought of as local expansions on $\mathbb{C}_z$ or $\mathbb{C}_s$. We focus on the singularity $\chi_1$ to illustrate the point. First, note that in the new variable $s_1$,
\begin{equation}
    y_B(s_1,z) = \frac{1/z}{s_1-1} + \frac{1/z}{s_1 - \frac{z^2+1}{z}} \,.
\end{equation}
The first expansion, about $s_1=1$, is thus
\begin{multline}
    y_B(s_1,z) = (1-s_1)^{-1}\Bigl[ -\frac{1}{z(z^2+z-1)} + (1-s_1)\frac{1}{(z^2+z-1)^2} \\ + \frac{z}{(z^2+z-1)^2}(1-s_1)^2 + \ldots \Bigr].
\end{multline}
Typically, the task is to obtain the coefficient functions in this expansion, or at least their leading-order (near $\Gamma_z$) constants, since after a simple re-scaling by factors of $\chi$ these will be the sought after coefficients $a_i(z)$ discussed in the previous sections. Let us conceal them:
\begin{equation}
    y_B(s_1,z) = (1-s_1)^{-1} \left( \tilde{a}_0(z) + \tilde{a}_1(z) (1-s_1) + \ldots \right)
\end{equation}
Alternatively we may expand around a zero of $\chi_1(z)$ (in this case $z=0$) to find
\begin{equation}
    y_B(s_1,z) = \frac{1}{z}\left((1-s_1)^{-1} + z s_1(1-s_1)^{-1} + z^2(1+s_1) + z^3 \frac{s_1(2-s_1^2)}{s_1-1} + \ldots \right) \,.
\end{equation}
Again, the coefficients here are typically unknowns, so let us write
\begin{equation}
    y_B(s_1,z) = \frac{1}{z}\left(\varphi_0(s_1) + z \varphi_1(s_1) + z^2 \varphi_2(s_1) + z^3 \varphi_3(s_1) + \ldots \right)
\end{equation}
Now the key point, as in the main body above, is that expanding at $s_1 = 0$ allows us to determine `initial conditions' in each $\varphi_i(s)$ by similarly expanding the perturbative terms. Focusing on $\varphi_0(s)$ we find
\begin{equation}
\begin{split}
    &\varphi_0(0) = -1 \,, \quad \varphi_0'(0) = 1 \\
    &\varphi_1(0) = -1 \,, \quad \varphi_1''(0) = 0 \,.
\end{split}    
\end{equation}
Suppose further that $\varphi_0(s),\varphi_1(s),\ldots$ were known to solve an ODE (this is what we see happens if $y_B(w,z)$ is the resurgent function associated to a singularly perturbed ODE) then we may obtain the coefficients of the $\tilde{a}_i(z)$ as expansions about $z=z_{\star}$. We explain this procedure in the ODE case in more detail in the following.
\end{example}

\begin{example}[An example with nested boundary layers]
There is a subtle assumption in the above discussion whereby, for a given $\chi(z)$ with $\chi(z^*) = 0$ for  $z_{\star} \in \Gamma_z$, it was possible to obtain initial data for the inner $\varphi_i(s)$ at $s=0$. It is possible to construct examples where this is not the case, and the inner-outer matching procedure becomes more subtle.

Consider, for example, three poles in the Borel plane at $w = \chi_i$ with $\chi_1(z) = z$, $\chi_2(z) = z^2$ and $\chi_3(z)=z^3$, together with a perturbative germ, $y = y_0(z) + \ep y_1(z) + \ep^2 y_2(z) + \ldots$. Notice that all three singularities coalesce to the origin, $w = 0$, as $z\to0$. We then re-scale according to the inner variable, 
\begin{equation}
    s_i = \frac{w}{\chi_i(z)} \,,
\end{equation}
for some choice of $i$. For the chosen $i$, the corresponding singularity will be at $s_i=1$ of the inner $\mathbb{C}_{s_i}$ plane, while the other two will move to either zero or infinity (or both) as $z \to 0$. This configuration is illustrated in figure \ref{fig:3eg}.

The complication is that if one were to work, for example, in the inner variable $s_1$, then the inner equation for $\varphi(s_1)$ will have a singular point at $s_1=0$, and we cannot determine the initial data from the perturbative germ in the way described above. The resolution is to first work with the inner variable, $s_3$, so that the singularities associated to $\chi_2$ and $\chi_3$ move to infinity of the $\mathbb{C}_{s_3}$ plane. We may then determine the solution near $w=\chi_3(z)$ since now, initial data for $\varphi(s_3)$ is determined by the perturbative germ at $s_3=0$. We may then consider the inner expansion for $\chi_2$ where now the initial data for the $\varphi(s_2)$ equations is determined from the singularity expansion associated to $\chi_3$---this singularity now moves to the origin in the $s_2$ variable. In this way, the inner-outer matching procedure may be iterated recursively. In this case one clearly needs to pay careful attention to the higher physical singularities $\Gamma^{\chi}_z$.

When viewed purely from the physical plane, $\mathbb{C}_z$, the application of exponential asymptotics for solving such problems with nested boundary layers can be difficult. Here, the above method is manifest in the Borel plane.
\end{example}

\begin{figure}
    \centering
    \includegraphics[scale=0.15]{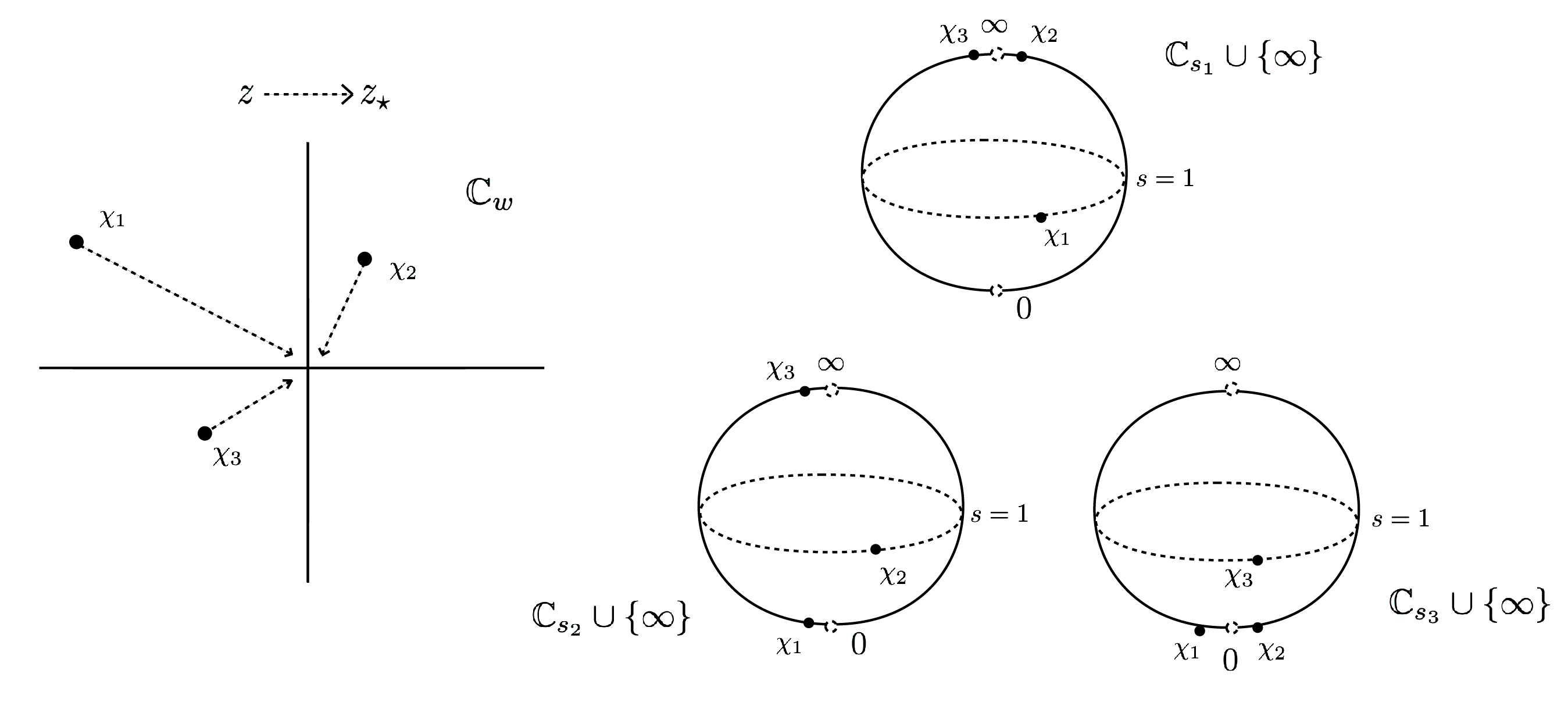}
    \caption{This illustrates an example of three singularities, $\chi_1$, $\chi_2$, and $\chi_3$ sharing a common $z_{\star} \in \Gamma_z$. (Left) As $z \to z_{\star}$ the singularities coalesce at the origin in $\mathbb{C}_w$. (Right) The compactified Borel planes, drawn on the Riemann sphere,  expressed in each Borel inner variable, $s$, are shown near $z_{\star}$.}
    \label{fig:3eg}
\end{figure}

\paragraph{Second-order linear ODEs.}
We now return to the second-order linear ODE of \cref{example_2ndorderkey}. Suppose that $y_B(w,z)$ is annihilated by the Borel operator
\begin{equation}
    \partial^2_z + P(z) \partial_z \partial_w + Q(z) \partial^2_{w},
\end{equation}
with $P$ and $Q$ meromorphic. We are thus studying perturbative solutions of the physical singular operator
\begin{equation}
    \mathscr{P} = \epsilon^2 \frac{d^2}{dz^2} + \epsilon P(z) \frac{d}{dz} + Q(z) \,.
\end{equation}
As explained earlier, inhomogenous terms for the physical operator yield initial data for the Borel PDE. For a given singularity, $\chi \in \Gamma_w(z)$, we may consider changing variables in the Borel PDE to the inner variable $s=w/\chi(z)$. In this variable the various relevant differential operators become:
\begin{equation}
\begin{aligned}
    \partial_z^2 &= s^2 \left( \frac{\chi'}{\chi}\right)^2 \partial^2_s + s \frac{(2(\chi')^2 - \chi'' \chi )}{\chi^2}\partial_s - 2s \frac{\chi'}{\chi} \partial_z \partial_s + \partial^2_z, \\
    \partial^2_w &= \frac{1}{\chi^2} \partial^2_s \,,\\
    \partial_w \partial_z &= - \frac{\chi'}{\chi^2}  \partial_s - s\frac{\chi'}{\chi^2}\partial^2_s + \frac{1}{\chi} \partial_s \partial_z \,.
\end{aligned}    
\end{equation}
The idea here is the following. Without loss of generality, suppose that $\chi(z)$ has an algebraic zero of order $\gamma$ at $z=z_{\star} = 0$ so that we may write locally $\chi(z) = \chi_0 z^{\gamma}$. Now from the fact $(\chi')^2 - P (\chi') + Q = 0$ we have that $P$ and $Q$ must have zeroes of (at most) orders $\gamma-1$ and $2\gamma-2$ respectively and we write them similarly to leading order:
\begin{equation}
    P(z) = P_0 z^{\gamma-1} \,, \quad Q(z) = Q_0 z^{2\gamma-2} \,.
\end{equation}
The consequence is that in the new variables, the Borel PDE is \textit{homogeneous}\footnote{If one assigns a multiplicative weight $\mathsf{w}=1$ to $z$ and $\mathsf{w}=-1$ to $\partial_z$ then every term has weight $\mathsf{w}=-2$.} in factors of $z$ and $\partial_z$. Acting on \eqref{eq:innerexpansion} with the operator in the new variables therefore couples $\varphi_i(s)$, $\varphi'_i(s)$ and $\varphi''_i(s)$ at the \textit{same} order in $z$ after expanding $\chi(z)$ and the coefficients $P(z)$ and $Q(z)$. We call the resulting equations satisfied by the $\varphi_i(s)$ the \textit{Borel inner equations}. For example, to lowest order we find
\begin{equation}
\begin{split}
    \partial_z^2 &= \gamma^2 s^2 z^{-2} \partial^2_s + s z^{-2}\gamma(\gamma+1)\partial_s - 2s \gamma z^{-1} \partial_z \partial_s + \partial^2_z \,, \\
    \partial_w^2 &= \chi_0^{-2} z^{-2\gamma} \partial_s^2 \,, \\
    \partial_w \partial_z &= \chi_0^{-1}(-\gamma z^{-\gamma-1}\partial_s - s \gamma z^{-\gamma-1}\partial^2_s + z^{-\gamma}\partial_s \partial_z) \,. 
\end{split}    
\end{equation}
Acting with these new `local' operators on \eqref{eq:innerexpansion}, we obtain the first inner ODE operator
\begin{equation}\label{eq:firstinner}
    \mathscr{P}_B^{\text{In.},0} = \left(\chi_0^2 \gamma^2 s^2 - \chi_0 P_0 \gamma s + Q_0\right)\frac{d^2}{ds^2} + \chi_0(\chi_0 \gamma (\gamma+1+2 \beta)s- P_0 (\gamma+\beta))\frac{d}{ds} + \chi_0^2 \beta(\beta+1)\,.
\end{equation}
Higher-order inner equations can be obtained similarly: one replaces $\beta \to \beta + k$ for the $k$th inner equation, and replaces $P_0$, $Q_0$ and $\chi_0$ accordingly.
We remind the reader that the inner equation depends on a particular $\chi(z)$ with a particular $z_{\star}$ and an associated $\beta$.
The discussion at the beginning of this section explains how to obtain initial values for this differential equation---namely the initial data may be read off from the perturbative initial values $y_0(z)$ and $y_1(z)$. 
We note that setting $s=1$, we find for the leading coefficient
\begin{equation}
    \chi_0^2 \gamma^2 - \chi_0 P_0 \gamma + Q_0 = 0 \,,
\end{equation}
where this is zero because of the algebraic equation satisfied by $\chi'$. Hence, as expected, the ODE has a leading singular coefficient at $s=1$.

We may verify that the inverse Borel transform (following the rules of lemma \ref{lemma:replacementrules}) of the above equation indeed recovers the more familiar `inner equation' in matched asymptotics posed in $\mathbb{C}_z$. In principle, one could recover the complete expansions of the singularly perturbed trans-series components, $a_i(z)$, by expanding all of the $\varphi_i(s)$ inner solutions. 

\begin{example}
The inner equation for the singularly perturbed first-order ODE $\epsilon y' + y = 1/z$ is given by
\begin{equation}
    (1-s)\varphi_0'(s) - \varphi_0(s) = 0 \,.
\end{equation}
Following the rules of lemma \ref{lemma:replacementrules} one may take the inverse transform to obtain, 
\begin{equation}
    \frac{dY}{dX} + Y = 1/X \,,
\end{equation}
where $Y:=\mathcal{B}^{-1} \varphi_0$ is the governing inner ODE in physical space. This is the familiar physical inner equation obtained by the inner rescaling $X=z/\epsilon$ in the original ODE.
\end{example}

\begin{example}
Consider a singularly perturbed forced harmonic oscillator
\begin{equation}
    \epsilon^2 y''(z) - y(z) = \epsilon^2 F(z) \,.
\end{equation}
In the Borel plane the corresponding operator is the wave equation
\begin{equation}
    \mathscr{P}_B = \partial^2_w - \partial^2_z = (\partial_w - \partial_z)(\partial_z - \partial_w) \,,
\end{equation}
together with the initial data $y_B(w=0,z) = -F(z)$. Since the operator factorises we may find $y_B$ in closed form. This is an interesting toy example that allows us to check a number of aspects of the formalism outlined above. Note however that in this example, the singularities in the Borel plane are doubled, i.e. lie at $w = \pm \chi(z)$, and the resurgence relations of section \ref{sec:background} need to be modified slightly to include functions with multiple power law singularities at the same radius. Such a result is straightforward to derive by studying the asymptotics of the integral in lemma \ref{lemma:hankelcoefficients}.
\end{example}

\subsection{Van-Dyke's matching rule in the Borel plane}
In applied matched asymptotics there is a well-known heuristic known as Van-Dyke's matching rule~\cite{van1975perturbation}. We now discuss how this procedure works by passing to the complex analytic side of \eqref{eq:correspondence}, independently of any singularly perturbed problem from which the asymptotic series arises. In fact, in the present framework, Van-Dyke's matching procedure is a straightforward consequence of the resurgent lemmas \ref{lemma:hankel} and \ref{lemma:hankelcoefficients} expressed in the holomorphic inner variable. 

Suppose that we obtain a local power series solution about $s=0$ of the inner solution $\varphi_0(s)$ in \eqref{eq:innerexpansion_pre}. We write this as
\begin{equation}
    \varphi_0(s) = \varphi_0^{(0)} + \varphi_0^{(1)} s + \varphi_0^{(2)} s^2 + \ldots
\end{equation}
In applications, one often obtains a (constant) recurrence relation for the $\varphi_0^{(n)}$. On the other hand, we may also write our singularity expansion in terms of the inner variable, $s$, giving
\begin{equation}
    y_B(s,z) = (1-s)^{-\alpha}(\tilde{a}_0(z) + (1-s)\tilde{a}_1(z) + (1-s)^2 \tilde{a}_2(z) + \ldots)
\end{equation}
where the set of functions $\tilde{a}(z)$ are related to the functions $a(z)$ in the ansatz \eqref{eq:ansatz} by
\begin{equation}
    \tilde{a}_n(z) = (-\chi(z))^{-\alpha+n}\ a_n(z) \,.
\end{equation}
We know that $y_B(s,z)$ has a singularity at $s=1$, by design, and we may extract the coefficients in the expansion around this singularity from the asymptotics of the coefficients $\varphi_0^{(n)}$. The resurgence relation tells us that if at large $n$ we find
\begin{equation}
    \varphi_0^{(n)} \to  \frac{\Gamma(n+\delta)}{\Gamma(n+1)} \left( C_0 + \frac{1}{n+\delta-1}C_1  +  \ldots \right) \,,
\end{equation}
where in fact $\delta = \alpha$ since this is the order of the singularity at $s=1$ and $C_i$ are constants, then we have
\begin{equation}
    \varphi_0(s) =  (1-s)^{-\delta}(C_0 + C_1(1-s) + \ldots ) 
\end{equation}
In this way, we determine the \textit{leading-order} terms at $z=z_{\star}$ of each $\tilde{a}_i(z)$ for $i=1,2,\ldots$ from the constants $C_i$. Similarly the sub-leading inner solutions $\varphi_i(s)$ determine the sub-leading (in $z_{\star}$) expansions of $\tilde{a}_i(z)$. An illustration of this idea is given in figure \ref{fig:vandyke}. We note that, as explained in this form, Van Dyke's rule is a property of asymptotic sequences with holomorphic $z$ dependence and has nothing to do with differential equations \textit{a priori}. It does, however, provide for us the sought-after initial data for $\tilde{a}_i(z)$ at $\Gamma_z$ in the case of singularly perturbed problems.

\begin{remark}
In practice, if one has a singularly perturbed ODE, it is easier to directly make an ansatz $\varphi_0(s) = (1-s)^{-\alpha}(\varphi_0^{(0)} + (1-s)\varphi_0^{(1)}+\ldots)$ however we present the above result from the perspective of the late terms since this is how one typically proceeds in matched asymptotic problems. 
\end{remark}
We now turn to a simple example.

\begin{figure}
    \centering
    \includegraphics{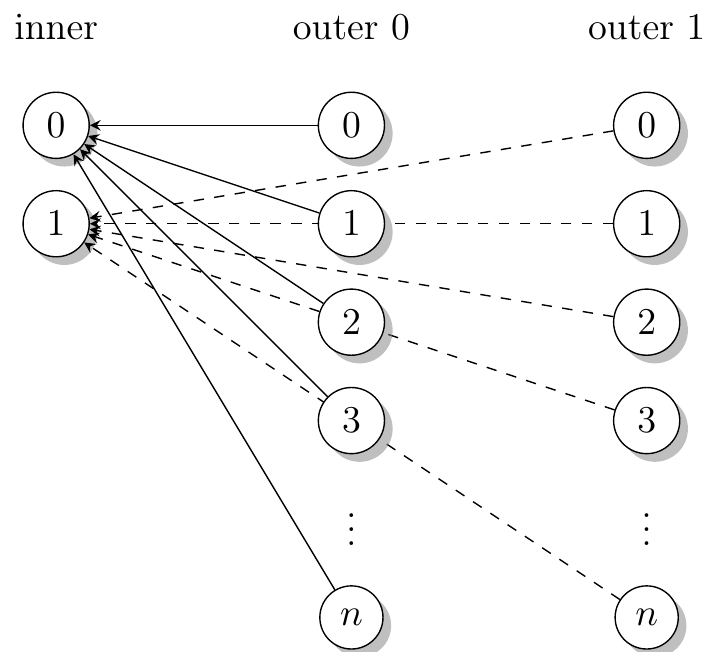}
    \caption{An illustration of Van Dyke's rule. The task is to relate the expansion of $y_B$ about $s = 1$ in $y_B = (1 - s)^{-\alpha}(\tilde{a}_0(z) + (1 - s) \tilde{a}_1(z) + \ldots)$ to the expansion about $z = z^*$ in $y_B = (z - z_*)^{-\beta}(\varphi_0(s) + z \varphi_1(s) + \ldots$. For example, solid lines correspond to the successive late terms of $\varphi_0(s)$ determining the leading terms of $\tilde{a}_i(z)$ in $z$ and dashed lines correspond to the late terms of $\varphi_1(s)$ determining the next-to-leading terms of the $\tilde{a}_i(z)$.}
    \label{fig:vandyke}
\end{figure}

\begin{example}
Consider the parametric resurgent function
\begin{equation}
    y_B(w,z) = \frac{1}{w-z}\frac{1}{w-(z^2+1)} \,,
\end{equation}
and let us study the singularity $\chi(z) = z$. In the inner variable, $s$, associated to $\chi$ we have
\begin{equation}
    y_B(s,z) = \frac{1/z}{s-1}\frac{1/z}{s-(z^2+1)/z} \,.
\end{equation}
In that case we have the first few inner solutions
\begin{equation}
\begin{split}
    \varphi_0(s) = -\frac{1}{1-s} \,,\quad \varphi_1(s) = \frac{s}{1-s}\,,\quad \varphi_2(s) = \frac{1+s^2}{1-s}\,, \ldots
\end{split}    
\end{equation}
The inner solutions have the late-term asymptotics
\begin{equation}
\begin{split}
    \varphi_{0,n} = -1 \,,\quad \varphi_{1,n} = 1\,,\quad \varphi_{2,n} = -2\,.
\end{split}    
\end{equation}
On the other hand, we may expand $\tilde{a}_0$ to give
\begin{equation}
    \tilde{a}_0(z) = -\frac{1}{z} + 1 - 2z + \ldots 
\end{equation}
This demonstrates the holomorphic Van-Dyke rule.
\end{example}

To conclude, if one obtains inner ODEs satisfied by the $\varphi_i(s)$ then one may obtain initial data for the $\tilde{a}_i(z)$ at $\Gamma_z$ by using the holomorphic Van-Dyke method outlined above and all components of the parametric trans-series may be obtained.

\begin{remark}
In the previous subsections, we have presented the ansatz method in parallel with the typical order of study in applied exponential asymptotics~\cite{chapman_1998_exponential_asymptotics} (late-order outer analysis before inner). However, we remark that, arguably, the procedure outlined in the previous two subsections should be the first step one undertakes when faced with a singularly perturbed problem: this is since the inner-outer matching procedure yields initial conditions for the $a_i(z)$ at $\Gamma_z$,  which may then be solved as initial-value problems.
\end{remark}

\section{Examples and discussion}\label{sec:applications}
In this section, we conclude with a worked example and a summary and discussion of the results presented with some proposals for future work.

\subsection{A worked example} \label{sec:workedexample}
We now illustrate how the framework outlined in the previous section may be applied with an example. We design this example so that, for simplicity, the leading singularity in the Borel plane is a pole. Consider the equation [cf. \eqref{eq:secondorderODE}]
\begin{equation}\label{eq:workedexample}
    \mathscr{P}y = \epsilon^2 y''(z) - 3 z \epsilon y'(z) + 2 z^2 y(z) = z \,,
\end{equation}
for which the corresponding operator in Borel space is
\begin{equation}
    \mathscr{P}_B = \partial^2_z - 3z \partial_{z}\partial_w + 2z^2 \partial_w^2 \,.
\end{equation}
Seeking a power series solution to $\mathscr{P}_B y_B = 0$ about $w=0$ of the form $y_B(w,z) = u_0(z) + w u_1(z) + \ldots$ we obtain the recurrence relation, for $n \geq 2$, 
\begin{equation}
    u_{n-2}(z) - 3 z (n-1) u_{n-1}(z) + 2z (n-1)(n-2)u_n(z) = 0 \,.
\end{equation}
As expected, this corresponds (up to the requisite factorial factors) with the perturbative solution to \eqref{eq:workedexample}. The forcing term of the ODE \eqref{eq:workedexample} provides us with a non-trivial power series about $w=0$ with leading terms
\begin{equation}
    u_0(z) = -\frac{3}{4z^3} \,,\quad u_1(z) = \frac{23}{8z^5} \,,\ldots
\end{equation}
We thus see that $\Gamma_z = \{0\}$. We now make the Borel singularity ansatz
\begin{equation}
    y_B(w,z) = (w-\chi(z))^{-\alpha}(a_0(z) + (w-\chi(z))a_1(z) + \ldots) \,,
\end{equation}
from which we learn two consistent singularities
\begin{equation}
    \chi_1(z) = -\frac{z^2}{2} + c_1 \, ,\quad \chi_2(z) = -z^2 + c_2 \,. 
\end{equation}
By the discussion in section \ref{subsec:boundarylayers} (namely initial conditions for $\chi$ are given by setting $\chi(z_{\star}) = 0$ for $z_{\star} \in \Gamma_z$) we further have $c_1=c_2=0$. Let us focus on the closer singularity $\chi_1(z) = -z^2/2$. The singularity ansatz provides us a differential equations for the coefficients, $a_i(z)$, i.e. [cf. \eqref{eq:a0an_gen}],
\begin{subequations}
\begin{align}
    z a_0'(z)-a_0(z) &= 0, \\
    (\alpha-n)(z a_n'(z) + a_n(z)) &= 0, \qquad \text{for $n \geq 1$,} \label{eq:workedexample4}
\end{align}
\end{subequations}
so that either we learn nothing\footnote{This is not a problem since the regular parts of $y_B(w,z)$ do not contribute to the perturbative trans-series expansion.} about $a_n(z)$ in the case $\alpha$ is an integer (that is a pole in the Borel plane) or we learn linear $a_n(z) = a_n z$. In summary, so far we have the local power series expansion
\begin{equation}
    y_B(w,z) = (w+z^2/2)^{-\alpha_1}(a_0(z) + a_1(z) (w+z^2/2) + \ldots ) \, ,
\end{equation}
we now find initial conditions at $z=0 \in \Gamma_z$ for the $a_i(z)$. Comparing to the perturbative germ at $w=0$
\begin{equation}\label{eq:workedexample2}
    y_B(w,z) = - \frac{3}{4z^3} + \frac{23}{8 z^5}w + \ldots
\end{equation}
we learn that $\alpha_1=2$. We now turn to inner-outer matching in order to determine the initial conditions. The inner variable associated to $\chi_1$ is $s_1 = -w/z^2$ and in the notation of section \ref{subsubsec:innerouter} we note the relevant parameters $\gamma = 2$, $P_0 = -3$, $Q_0=2$ and $\beta = 3$. The first inner equation thus reads
\begin{equation}
    (s-1)(s-2)\varphi_0''(s) + \frac{1}{2}(9s -15) \varphi_0'(s) + 3 \varphi_0(s) = 0 \,,
\end{equation}
and equation \eqref{eq:workedexample2} tells us that the initial data is
\begin{equation}
    \varphi_0(0) = -\frac{3}{4} \,, \quad \varphi_0'(0) = \frac{23}{16} \,.
\end{equation}
We seek a power series solution about the singular point $s=1$ and find
\begin{equation}\label{eq:workedexample3}
    \varphi_0(s) = (1-s)^{-2} \left( -\frac{\sqrt{2}}{2} - \frac{\sqrt{2}}{16}(1-s)^2 + \ldots \right) \,,
\end{equation}
Alternatively (and more closely to the typical matched asymptotic analysis) one may seek a solution of the form $\varphi_0(s) = \sum_{n=0}^{\infty}\varphi_{0,n}s^n$ and solving the resultant (numerical) recursion relation to find at leading order\footnote{Further terms may be found by expanding in $1/n$ and using lemma \ref{lemma:hankelcoefficients} to relate late-terms to saddles in the usual way.}
\begin{equation}
    \varphi_{0,n} \to -\frac{\sqrt{2}}{2} n + \ldots  \,.
\end{equation}
The general theory of section \ref{subsec:ansatz} finally tells us that we now have initial data for the $\tilde{a}_0(z)$ and $\tilde{a}_1(z)$, and in turn the $a_0(z)$ and $a_1(z)$ (in this case these are the only coefficients that are not part of the regular part). However, we note the $(1-s)^{-1}$ term of \eqref{eq:workedexample3} is zero which implies that we have to consider higher inner equations in order to determine the initial data for $\tilde{a}_1$. Proceeding, solving the differential equations \eqref{eq:workedexample4} gives $a_0(z) = a_{0,0}z$ we may write the inner expansion as
\begin{equation}
    y_B(s,z) = (1-s)^{-2}(\frac{4}{z^3}a_{0,0} + \frac{2}{z} a_{1,0} + \ldots)
\end{equation}
We deduce from the above that $a_{0,0} = -\frac{\sqrt{2}}{8}$. Now we must consider the third inner equation for $\varphi_2(s)$ in order to determine $a_{1,0}$. Following the method of section \ref{subsubsec:innerouter} this equation reads
\begin{equation}
    (s-1)(s-2)\varphi_2''(s) + \frac{1}{2}(5s-9) \varphi_2'(s) + \frac{1}{2}\varphi_2(s) = 0 \,.
\end{equation}
However, comparing with the perturbative data, we learn that $\varphi_2(0) = \varphi_2'(0)=0$ so that in fact $a_{1,0}=0$. We now have all the components of our local expansion near $\chi_1$:
\begin{equation}
    y_B(w,z) = (w+z^2/2)^{-2}(-\frac{\sqrt{2}}{8} z) + \text{reg.}
\end{equation}
From \cref{subsec:ansatz}, this is sufficient to determine the Stokes phenomena to all orders\footnote{The fact we have a double pole means the asymptotics terminate and there are no further corrections.} in $\epsilon$. Namely, crossing the Stokes' line $l_1$ turns on the term
\begin{equation}
    \frac{2 \pi \im}{\epsilon}\frac{\sqrt{2}z}{8} \e^{-z^2/(2 \epsilon)}  \,.
\end{equation}
The analysis for the second singularity is similar except one finds a branch point rather than a pole in the Borel plane and the fluctuations do not truncate as above. We illustrate the Stokes lines for this problem in \cref{fig:workedexample_stokes}.

\begin{figure}[htb]
    \centering
    \includegraphics{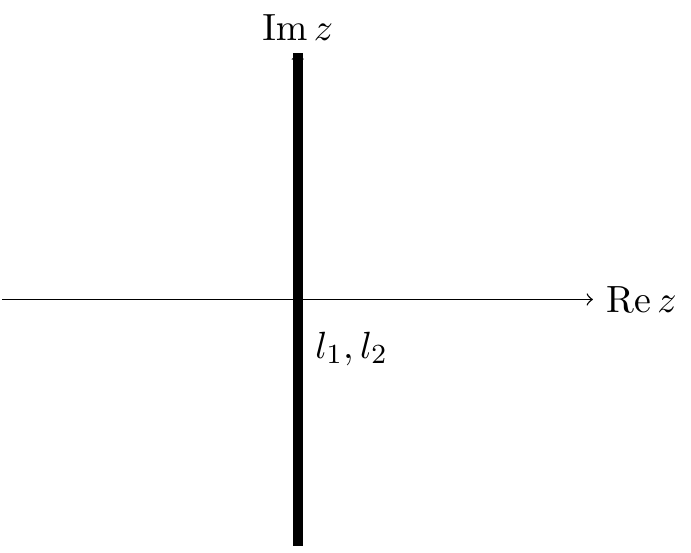}
    \caption{Stokes graph, $l = l_1 \cup l_2$, from \cref{sec:workedexample}.}
    \label{fig:workedexample_stokes}
\end{figure}

\subsubsection{Pad\'{e} approximants for singular problems}
The Borel germ $y_B(w,z)$ has a finite radius of convergence up to $w=\chi(z)$. Pad\'e approximants may be used as a heuristic to study the analytic continuation of the germ of $y_B(w,z)$ from a finite number of perturbative coefficients, and hence allow a useful glimpse into the whole Borel plane. In general, a Pad\'e approximant is a rational approximation to a locally holomorphic function $f(x)$, which may be computed from the first $N$ coefficients about an analytic point. Given $f(x) = a_0 + a_1 x + \ldots a_{2N} x^{2N}$ we solve the algebraic system resulting for $\{r_j\}$ and $\{ x_j\}$ from equating derivatives at zero on either side of
\begin{equation}
    a_0 + a_1 x + \ldots + a_{2N} x^{2N} = \sum_{j=0}^N \frac{r_j}{x-x_j} =: P_{N:N-1}(x) \,.
\end{equation}
The unique solution then determines the off-diagonal Pad\'e approximant $P_{N:N-1}$. Whilst uniform convergence of $P_{N:N-1}$ to $f(x)$ for large $N$ is far from guaranteed \cite{lubinsky2003rogers}---even pointwise convergence fails in general---Pad\'e approximants often give an unreasonably good picture of the analytic continuation of $f(x)$. We refer to \cite{costin2020uniformization, Costin:2021bay, Costin:2020hwg} for a review of the state of the art of Pad\'e approximant methods and note here only some heuristics. Returning to the example of \cref{sec:workedexample}, for fixed $z$ we may compute the Pad\'e approximant of the germ
\begin{equation}
    y_B(w,z) = u_0(z) + w u_1(z) + w^2 u_2(z) + \ldots
\end{equation}
giving $z$-dependent $\{r_j(z)\}$ and $\{x_j(z)\}$. It is then interesting to observe the singularities cross the contour as $z \in \mathbb{C}_z$ is varied; we present some representative plots for our worked example in \cref{fig:workedexample}. We draw particular attention to the coalescing poles that Pad\'e typically assigns to branch cuts---these are associated with the branch point singularity at $\chi_2(z) = -z^2$ in the present example.

\begin{figure}[htb]
    \centering
    \includegraphics{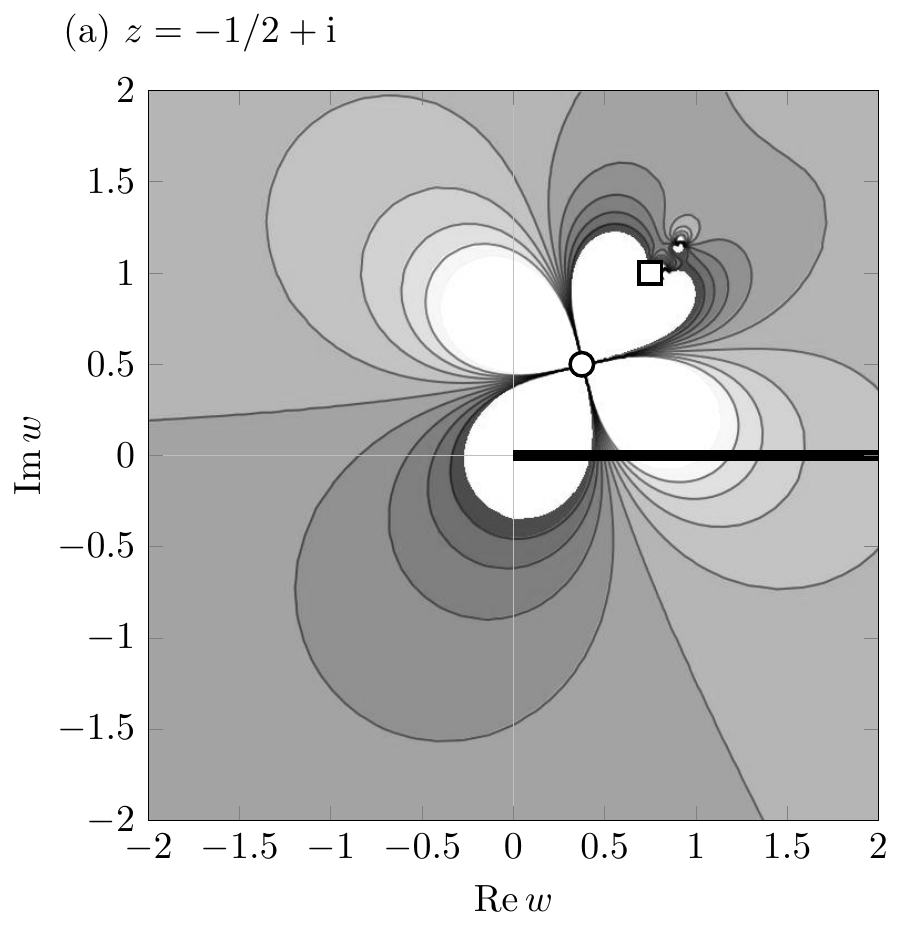}
    
    \includegraphics{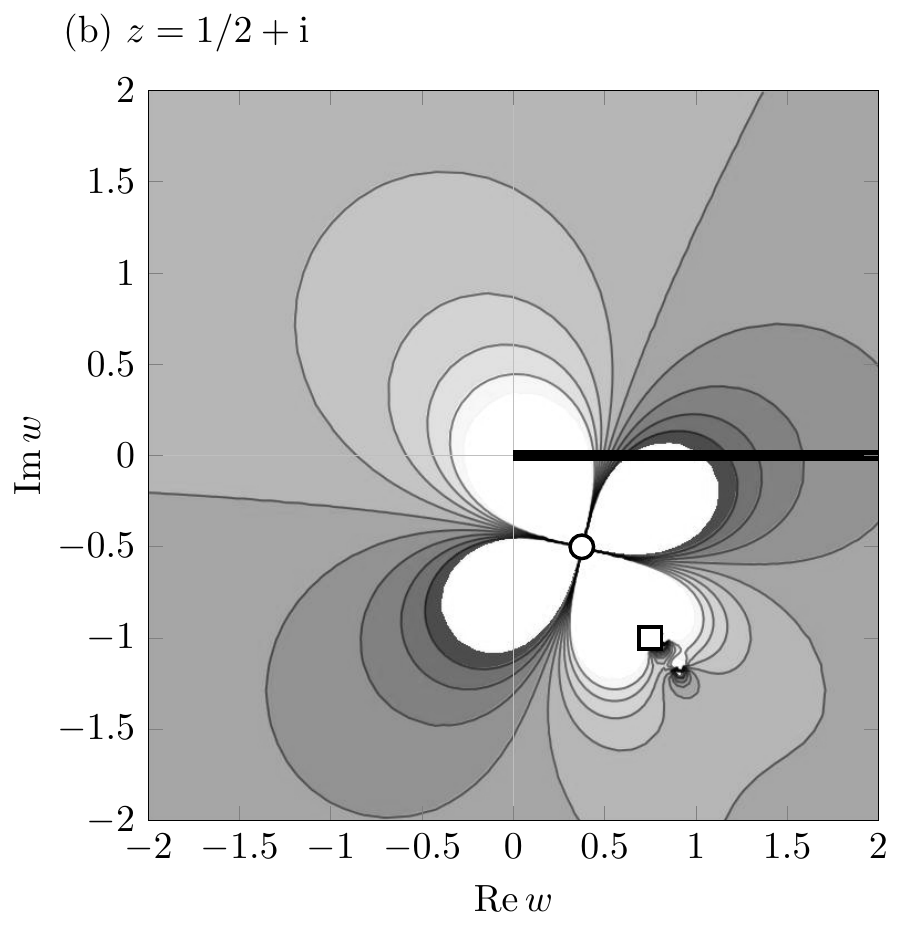}
    \caption{Pad\'e approximants for the worked example in \cref{sec:workedexample}. We compute the $[20:21]$ off-diagonal Pad\'e with $200$ perturbative coefficients in Mathematica for the two cases of (a) $z = -1/2 + \im$ and (b) $z = 1/2 + \im$. Contours correspond to $\Re y_B$. There are two singularities at $w = \chi_1(z) = -z^2/2$ (circle) and $w = \chi_2(z)= -z^2$ (square). As the Stokes line of \cref{fig:workedexample_stokes} is crossed, two exponentials switch on due to intersection of $\chi_{1,2}$ with the Borel integration contour (thick).}
    \label{fig:workedexample}
\end{figure}

\begin{remark}
We note that for any linear singularly perturbed ODE there is an algorithmic procedure to recover Pad\'e plots around any singularity in the Borel plane. One may first compute the Pad\'e approximant about $w=0$ in the way discussed above. Secondly, around a singularity $\chi$ the local coefficients $a_i^{\chi}(z)$ may be determined to an arbitrary order in $(w-\chi)$ by solving the recursive differential equations \eqref{eq:secondorderrecursive} and inner equations \eqref{eq:firstinner}. One may then compute a Pad\'e approximant around each $\chi$ to see the Borel plane expansion around each singularity. This is a useful heuristic to detect higher-order singularities on different Riemann sheets (for particular regions of $z$), since these must appear only in the expansions around $\chi$ but not $w=0$.
\end{remark}

\subsection{First-order equations and coalescence of singularities}\label{subsec:firstorder}

First-order singularly perturbed equations are exactly soluble in the Borel plane.\footnote{This is equivalent to the fact that such equations are soluble (using an integrating factor) as an integral.} Let us write a first-order differential equation as\footnote{The forcing term may be changed to $F(z)$ with no $\epsilon$ pre-factor, but then one must be careful to take the initial data for the Borel PDE as $y_1(z)$ rather than $F(z)/G(z)$. As discussed in section \ref{subsec:boreltransform} the Borel transform then `ignores' $y_0(z)$. We choose instead here to begin our expansion for $y(z,\epsilon)$ at $\mathcal{O}(\epsilon)$.}
\begin{equation} \label{eq:1storder_example}
    \mathscr{P}y = \epsilon y'(z) + G(z) y(z) = \epsilon H(z) \,,
\end{equation}
where $H(z)$ and $G(z)$ are assumed to be meromorphic for simplicity. According to lemma \ref{lemma:replacementrules}, the associated partial differential operator in the Borel plane is
\begin{equation}
    \mathscr{P}_B = \partial_z + G(z) \partial_w
\end{equation}
together with the initial data $y_B(w=0,z) = H(z)/G(z)$. The PDE may then be integrated to give
\begin{equation}\label{eq:firstordersolution}
    y_B(w,z) = \frac{H(\chi^{-1}(-w+\chi(z)))}{G(\chi^{-1}(-w+\chi(z)))} \,,
\end{equation}
where $\chi(z) = \int^z_{z_{0}} \de{s} \, G(s)$. In fact, we may observe that the solution is independent of $z_0$. 

\begin{remark}
We note the two types of singularities discussed in section \ref{sec:singularperturbation} are manifest here. Those arising from poles of $H(z)$ and those arising at points where $G'(z)=0$ which, by the inverse function theorem, coincide with points where $\chi(z)$ fails to be invertible.
\end{remark}

\paragraph{Coalescing singularities.}
There are unique complexities in cases where the forcing term, via \eqref{eq:1storder_example}, may depend on $\epsilon$, and where singularities coalesce as $\ep \to 0$. For example, let us consider 
\begin{equation} \label{eq:1st_coalesce}
    \epsilon y'(z) + y(z) = \epsilon H(z; \, \epsilon) \,.
\end{equation}
According to the replacement rules of lemma \ref{lemma:replacementrules}, in Borel space the problem reads
\begin{equation}
    \partial_z y_B(w,z) + \partial_w y_B(w,z) = H_B(w,z)
\end{equation}
where $H_B(w;z)$ is the Borel transform of $H(z; \, \ep)$. We expand $H(z; \, \epsilon)$ as a power series in $\epsilon$:
\begin{equation}
    H(z; \, \epsilon) = H_0(z) + H_1(z) \epsilon + \ldots \,.
\end{equation}
Note that typically the forcing term, $H(z;\epsilon)$, thought of as a function of $\epsilon$, will in fact have a finite radius of convergence. According to the results of section \ref{sec:background}, this means that the Borel transform $H_B(w,z)$ will be entire. The initial data is $y_B(w=0,z) = H_0(z)$, and we may integrate the Borel PDE to give the solution as an explicit integral,
\begin{equation}\label{eq:inhomogeneousintegral}
    y_B(w,z) = H_0(w-z) + \int_{\gamma_z} \de{t}  \,H_B(w-z+t,t) \,,
\end{equation}
where the integration cycle $\gamma_z$ ends at $t=z$ (note that if $H$ is a convergent series in $\ep$, then $H_B$ is non-singular, as per the above remark, and $z$ may be any point in $\mathbb{C}$).
\begin{example}
From \eqref{eq:1st_coalesce}, let us consider a forcing term with two singularities in $\mathbb{C}_z$ that coalesce to the origin as $\epsilon \to 0$:
\begin{equation}
H(z; \, \epsilon) = \frac{z}{(z+ \Delta \epsilon )(z- \Delta \epsilon)} \,,
\end{equation}
where $\Delta$ is an $O(1)$ real constant. Computing the associated Borel transform we find
\begin{equation}
    H_B(w,z) = \frac{\Delta}{z^2} \sinh{\frac{\Delta w}{z}} \,,
\end{equation}
which, as expected, is entire on $\mathbb{C}_w$. We may now compute the integral solution to the Borel PDE \eqref{eq:inhomogeneousintegral} to give
\begin{equation}
    y_B(w,z) = -\frac{\cosh \frac{\Delta w}{z}}{w-z} \,.
\end{equation}
Despite there being two singularities at finite $\epsilon$ in the forcing term there is only \textit{one} singularity in $\mathbb{C}_w$ at $w=z$, in this case a simple pole. The associated single Stokes line is then the real axis. From the inversion formula, we see that upon crossing the Stokes line, the exponentially-small contribution 
\begin{equation}
    \int_{\mathcal{H}_{\chi}} \de{w}  \, \e^{-w/\epsilon}\, \frac{\cosh \frac{\Delta w}{z}}{w-z} 
    = - 2 \pi \i (\cosh{\Delta})  \e^{-z/\epsilon} \,.
\end{equation}
is switched on. 

We note that the above Borel-plane procedure can be contrasted to \textit{e.g.}~\cite{trinh2015exponential} where applied exponential asymptotics techniques can struggle to develop expansions to such coalescing singularities problems. Here we see that when viewed via the lens of Borel transforms, the structure of the problem is manifest (namely the presence of only one Borel singularity in this case) and the derivation of the exponentially small terms is relatively straightforward. 
\end{example}

\subsection{Beyond power law singularities}\label{subsec:generalsings}
So far in this work we have considered holomorphic functions $f: \mathbb{C} \to \mathbb{C}$ with relatively tame singularities of power type [cf. \eqref{eq:elementarysingularity}]. More exotic singularities may arise as Borel-plane solutions to apparently simple singularly perturbed differential equations. Consider, for example, the equation
\begin{equation}\label{eq:funky1}
    \epsilon y' + e^z y = \epsilon z \,,
\end{equation}
so that $G(z) = e^z$, $H(z)=z$ and $\chi(z) = e^z$. Then, from the expression \eqref{eq:firstordersolution}, the solution in the Borel plane is given by
\begin{equation} \label{eq:yB_beyond}
    y_B(w,z) = \frac{\log(-w+e^z)}{-w+e^z} \,.
\end{equation}
We consider initial conditions for \eqref{eq:funky1} such that the inverse Borel transform with the perturbative contour, $\gamma_0 = (0,\infty)$, solves the equation along the real line:
\begin{equation}
    y(z;\epsilon) = \int_{(0,\infty)} \de{w}  \,e^{-w/\epsilon}\frac{\log(-w+e^z)}{-w+e^z} \,.
\end{equation}
The asymptotics in different regions of the complex plane $\mathbb{C}_z$ may then be read off using the procedure discussed in section \ref{subsec:boreltransform}.

Notice that from \eqref{eq:yB_beyond}, the Borel transform has a singularity of the form $y_B \sim \log(-s)/(-s)$ under $s = w - \chi(z)$. Hence this is not in the class of elementary singularities discussed in \eqref{eq:elementarysingularity}. We consider the coefficient integral in lemma \ref{lemma:hankelcoefficients} to leading order in $n$ and find
\begin{equation}
    \left(y_B(z)\right)_n = \frac{1}{2 \pi \i}\frac{1}{n \chi(z)^n} \int_{\mathcal{H}_0} ds\, \e^{-s} \phi\left(\frac{\chi(z)s}{n}\right)
    \sim \frac{\log(n)}{\chi(z)^{n+1}} \,.
\end{equation}
The Stokes line, $l \subset \mathbb{C}_z$, for this problem is given by 
\begin{equation}
    l: \, \Im e^z = 0, \quad \Re e^z > 0 \,,
\end{equation}
and the contribution to the trans-series across this line is given by
\begin{equation}
    \int_{\mathcal{H}_\chi} \de{w} \, \e^{-w/\epsilon} \frac{\log(-w+\e^z)}{-w+\e^z}\,.
\end{equation}
We may use the shortcut discussed below corollary \ref{cor:stokesswitching} to evaluate the exponentially small term to leading order. In the notation of that subsection, in this case we have $g(n,z) = n\log(n)/\chi$ so that the contribution across $l$ is
\begin{equation}
    \sim 2 \pi \i \log (\epsilon ) \e^{-\chi(z) / \epsilon} = 2 \pi \i \log (\epsilon ) \e^{-\e^{z}/ \epsilon} \,.
\end{equation}

\paragraph{General singularities.}
In the present framework, it is possible to produce first-order differential equations that have prescribed singularities (in the Borel plane) by reverse engineering the solution \eqref{eq:firstordersolution}. From this, we see that the standard factorial-over-power asymptotics of the form \eqref{eq:ffactpow_1} is far from universal. In particular, suppose that one wishes to construct an ODE with the Borel singularity form, $y_B \sim \phi(s)$ near $s = 0$. Then from the general solution \eqref{eq:firstordersolution}, we see that we must solve the functional relation
\begin{equation}
    F(s) = \phi\left(\int^s \de{t}\, G(t) \right)G(s) \,,
\end{equation}
for $F(s)$ and $G(s)$. Such solutions are clearly not unique and we give some example representatives in table \ref{tab:examplesings} below. It appears unlikely that, even for the restricted examples of table \ref{tab:examplesings}, such singularities form a resurgent algebra in the sense of section \ref{subsec:boreltransform}---for example it is unclear, at least to the authors, how one would close the algebra under convolution. However, in linear ODEs where there is a relatively simple singularity structure in the Borel plane this is no barrier to studying the associated Stokes phenomena and examples of singularities beyond the power law form do arise in physical examples~\cite{yyanis_thesis}.

\begin{table}[htb]
\begin{tabular}{c|c|c|c}
    \textbf{Borel singularity $\phi(s)$} & \textbf{Representative ODE} & \textbf{Late terms} & \textbf{Stokes contribution} \\
    $\log(s)$ & $\epsilon y' + e^z y = \epsilon z e^z$ & $\frac{1}{n \chi(z)^n}$ & $2 \pi i \epsilon e^{-\chi / \epsilon}$ \\ 
    
    $\log(s)^2$ & $\epsilon y' + e^z y =\epsilon z^2 e^z$ & $\frac{2 \log (n)}{n \chi^n}$ & $4 \pi i \epsilon \log (\epsilon) e^{-\chi / \epsilon}$ \\ 
    
    $\frac{\log(s)}{s}$ & $\epsilon y' + e^z y = \epsilon z $ & $\frac{\log(n)}{\chi^{n+1}}$ & $2 \pi i\log \epsilon e^{-\chi / \epsilon}$ \\ 
    
    $\text{Ei}^{-1}(s)$ & $\epsilon y' + z^{-1}e^z y = \epsilon$ & ? & ? \\ 
\end{tabular}
\caption{\label{tab:examplesings}Some example first-order singularities.} 
\end{table}

In some cases, the asymptotics of the coefficients associated to a singularity are not known. The seemingly innocuous equation in the last row of table \ref{tab:examplesings} is
\begin{equation}\label{eq:treeexample}
    \epsilon y' + \frac{e^z}{z} y = \epsilon \,.
\end{equation}
The Borel plane solution has singularities coinciding with those of inverse of the exponential integral function. To the authors' knowledge, little is known about the asymptotics of the associated coefficients and therefore it is not possible at present to determine the Stokes' switching term for this example.

Modifications to the basic factorial over power ansatz \eqref{eq:ffactpow} have arisen in physical applications---see, for example, \cite{trinh2015exponential}. The modification is explained in the present context as arising due to the presence of a singularity in the Borel plane that is not of power-law form.

\subsection{Discussion} \label{sec:discuss}
We conclude with a discussion of some directions for future research.

\paragraph{Singularly perturbed partial differential equations.}
In the present work we studied holomorphic trans-series with one additional holomorphic parameter $z \in \mathbb{C}$:
\begin{equation}
    y(z,\epsilon) = \sum_{\chi}e^{-\chi(z)/\epsilon} y^{\chi}(z,\epsilon) \,,
\end{equation}
we have discussed the correspondence with Borel functions $y_B(w,z)$ and explained how a differential operator $\mathscr{P}_B$ yields differential equations (in $z$) for the trans-series components. The interplay of singularities $\Gamma_w(z) \subset \mathbb{C}_w$ in the Borel plane and $\Gamma_z \subset \mathbb{C}_z$ in the physical plane allowed us to find initial conditions for the trans-series differential equations at points in $\Gamma_z$ via the process of inner-outer matching.

Much of this perspective neatly lifts to holomorphic trans-series with two parameters $z_1,z_2 \in \mathbb{C}_{z_1} \times \mathbb{C}_{z_2}$ of the form
\begin{equation}
    y(z_1,z_2,\epsilon) = \sum_{\chi}e^{-\chi(z_1,z_2)/\epsilon} y^{\chi}(z_1,z_2,\epsilon) \,,
\end{equation}
with associated two-parameter Borel germs $y_B(w;z_1,z_2)$. A linear singularly perturbed PDE $\mathscr{P}$ similarly yields a linear Borel PDE $\mathscr{P}_B$ from which now \textit{partial} differential equations may be derived for the trans-series components $\chi(z_1,z_2)$ and $a^{\chi}(z_1,z_2)$. We now have to define \textit{initial data} for these equations along \textit{curves} in $\mathbb{C}_{z_1} \times \mathbb{C}_{z_2}$

In the two variable case the inner-outer matching has more structure since the physical singularity set now has some geometry and defines a set of co-dimension one curves $\Gamma_z \subset \mathbb{C}_{z_1} \times \mathbb{C}_{z_2}$. The difference with the one variable case is that we now have to take our coordinate choice on $\mathbb{C}_w \times \mathbb{C}_{z_1} \times \mathbb{C}_{z_2}$ more seriously. The operator $\mathscr{P}_B$ yields a set of bicharacteristics now in $T^*(\mathbb{C}_{z_1} \times \mathbb{C}_{z_2} \times \mathbb{C}_w)$, projecting these to physical space $\mathbb{C}_{z_1} \times \mathbb{C}_{z_2}$ defines a foliation and gives curves along which singularities propogate (the singulant), we may choose a local coordinate $q$ adapted to this foliation. For an orthogonal coordinate, we write $p$ for a local coordinate on a connected component of $\Gamma_z$. The inner-outer matching is now concerned with local data near $p=0$ with a direction $q$ surviving (in contrast with the inner-outer matching being local to points in the one parameter case). In this way, we find a single parameter resurgent problem along $q$ which we must solve using, \textit{e.g.} the methods of the present work, to determine \textit{initial data} (as opposed to initial conditions at points) for the trans-series components---again along $\Gamma_z$. We leave a more thorough exploration of this perspective to future work.

\paragraph{Parametric analytic combinatorics.}
The world of analytic combinatorics provides a wealth of examples of unusual singularities and late term asymptotics (we refer the reader to \cite{flajolet2009analytic} for an excellent review and introduction to the topic). Many of the unusual singularity examples discussed above arise as counts of combinatorial objects known as \textit{trees}; whereby the perturbative germ gives a generating function and the coefficients of the Borel transform give the exponential generating function for these objects. Indeed, note that the coefficients of the Borel germ in each example of table \ref{tab:examplesings} (when simultaneously expanded about suitable points in $z$) are positive integers so that at least the possibility of counting something is present. In fact, in the context of singular perturbation theory (in the presence of an extra parameter $z$ together with the Borel variable), we find that the dual expansion in $w$ and $z$ often gives a refined (by, for example, leaves) count of trees that does not appear to have been noted in the analytic combinatorics literature.

\paragraph{Higher-order Stokes phenomena.}
In contrast to constant resurgent problems, singularly perturbed problems have the additional feature that Stokes lines naively associated to singulants, $\chi(z)$, may not always be `active'. This structure has been explored in a number of examples and perspectives \cite{howls2004higher,howls2012exponentially,chapman2007shock,aoki2005virtual}. 

The phenomena may be explained as follows. Suppose that $y_B(w,z)$ is a germ with Borel singularities, $\Gamma_w(z)$, and physical boundary layers, $\Gamma_z$, at $w=0$. We have discussed already that it is possible that the fluctuations $a^{\chi}(z)$ around a singularity, $\chi$, may yield an additional set of physical singular points $\Gamma^{\chi}_z$. In the applied exponential asymptotics terminology, these extra singularities may `drive divergence' and yield an additional Borel singularity, $w = \tilde{\chi}$, and corresponding term, $e^{-\tilde{\chi}/\epsilon}$, in the trans-series expansion. This is sometimes referred to as a `second-generation singularity'~ \cite{chapman2005exponential}. 

However, if the original singularity, $\chi$, is a branch point of $\Sigma_z$ it is possible that $\tilde{\chi}$ could not be found by the original trans-series/Borel singularity ansatz. Now, in the context of singular perturbation theory, both $\chi(z)$ and $\tilde{\chi}(z)$ may depend on $z$, and an additional kind of Stokes phenomena is possible. The naive Stokes line assigned to $\tilde{\chi}$:
\begin{equation}
    l'_{\tilde{\chi}}: \Re \tilde{\chi} > 0 \quad \Im \tilde{\chi} = 0\,,
\end{equation}
may not always be `active' since for some region(s) in $\mathbb{C}_z$, $\tilde{\chi}(z)$ may lie on a different sheet of $\Sigma_z$ to the perturbative contour $\gamma_0 = (0,\infty)$ and no `contour snagging' occurs. In this way, the true Stokes line $l_{\tilde{\chi}} \subset l'_{\tilde{\chi}}$ may be truncated, or removed entirely.

In the present work we see that extra singularities can be viewed from three dual perspectives. The discussion of section \ref{sec:background} explains that an additional divergence in the $1/n$ expansion of the form $\tilde{\chi}^{n+\alpha_{\tilde{\chi}}}$ may be associated to an additional component of the trans-series $e^{-\tilde{\chi}/\epsilon}$. Equivalently, as discussed above, the additional $\tilde{\chi}$ arises an additional singular points in the (shifted) power series expansion near $\chi$ of the Borel transform:
\begin{equation}
    y_B(t,z) = t^{-\alpha_\chi} \sum_{i \ge 0} a^{\chi}_i(z) t^i \,,
\end{equation}
where the series defined by the $a^{\chi}_i(z)$ may have a finite radius of convergence up to an additional $\tilde{\chi}$.
Higher-order Stokes phenomena can be found in relatively simple second order inhomogeneous linear equations. For example,
\begin{equation}
     \epsilon^2 y''(z) + \epsilon (1+z) y'(z) + z y(z)=1, 
\end{equation}
has the feature. \'Ecalle's alien calculus can be viewed as encoding the singularity structure of the analytic continuation of the Borel germ $y_B$. In the context of singular perturbation theory, the alien derivatives $\Delta_\chi(z)$ now depend on a holomorphic parameter $z \in \mathbb{C}_z$. Since the higher-order Stokes phenomena is intimately tied to the structure of $\Sigma_z$ it would be interesting to develop a holomorphic form of the bridge equation where we hope to give a definition of higher-order Stokes phenomena from the perspective of a (parametric) alien calculus.

\section{Acknowledgements}
The authors would like to thank the Isaac Newton Institute for Mathematical Sciences, Cambridge, for support and hospitality during the programme `Applicable resurgent asymptotics: towards a universal theory' where work on this paper was undertaken. It is also a pleasure to thank Dan Zhang, Hunter Dinkins, Nick Dorey, Shrish Parmeshwar, Juan Villarreal and Karen Yeats for many stimulating discussions. SC is particularly grateful to Akira Shudo and Gerg\H{o} Nemes for many interesting discussions and hospitality at Tokyo Metropolitan University while this work was completed. This work was supported by EPSRC grant no. EP/R014604/1. 

\appendix

\section{Gamma function}\label{sec:gammaappendix}
We recall some well-known properties of the $\Gamma(z)$ function (see \textit{e.g.} Abramowitz \& Stegun~\cite{abramowitz1964handbook} and DLMF~\cite{NIST:DLMF}).
\paragraph{Definition.}
The $\Gamma$-function is a meromorphic function which for $\Re z > 0$ is defined by an integral
\begin{equation}\label{eq:gammafunction} 
    \Gamma(z) = \int_{(0,\infty)} dx \,e^{-x}x^{z-1} \,.
\end{equation}
The analytic continuation of $\Gamma(z)$ to the whole of $\mathbb{P}^1$ has no zeros and simple poles at $\Gamma(z=-n)$ for $n$ a positive integer. The reciprocal function $1/\Gamma(z)$ thus has simple poles at the negative integers. An important property of the $\Gamma$ function is
\begin{equation}
    \Gamma(z+1) = z \Gamma(z) \,,
\end{equation}
which together with $\Gamma(1) = 1$ implies that $\Gamma(n) = (n-1)!$ for $n$ a positive integer.

\paragraph{Reciprocal Gamma function.}
Hankel contour integrals play an important role in the present work. Recall the Hankel contour representation of the reciprocal $\Gamma$ function. 

\begin{lemma}
The reciprocal Gamma function may be expressed as a contour integral as follows
\begin{equation}
    \frac{1}{\Gamma(\alpha)} = -\frac{1}{2 \pi i}\int_{\mathcal{H}_0} dt \,(-t)^{-\alpha}e^{-t} \,,
\end{equation}
where $\alpha \in \mathbb{C}\backslash \mathbb{Z}_{\le 0}$ and the Hankel contour centred on $0$, $\mathcal{H}_0$, is any contour homotopic to
\begin{equation}
    \mathcal{H}_0 = \{ t = z+i \, : \, z \ge 0 \} \cup \{ t=- e^{i \theta} \, : \, -\pi / 2 \le \theta \le \pi /2 \} \cup \{ t = z-i \, : \, z \ge 0 \}
\end{equation}
oriented anti-clockwise.
\end{lemma}
\begin{proof}
Consider the contour $\delta \mathcal{H}_0$ in the limit that $\delta \to 0$. The contribution from the middle part of the contour above vanishes and one may use the reflection identity
\begin{equation}
    \Gamma(z)\Gamma(1-z) = \frac{\pi}{\sin \pi z} \,,
\end{equation}
to obtain the result.
\end{proof}
Differentiating the above identity with respect to $\alpha$ yields the following: 
\begin{equation}\label{eq:logintegral}
    \int_{\mathcal{H}_0} dt \, e^{-t}(-t)^{-\alpha} \left(\log(-t)\right)^k = 2 \pi i(-1)^k \left.\frac{d^k}{d x^k} \frac{1}{\Gamma(x)} \right\rvert_{x=\alpha}.
\end{equation}

\paragraph{Asymptotics.}
The following identity follows from Sterling's formula \cite{tricomi1951asymptotic}:
\begin{equation}\label{eq:gammaasymptotics}
    \frac{\Gamma(n+\alpha)}{\Gamma(n+1)} = n^{\alpha-1} \left( 1+ \frac{\alpha(\alpha-1)}{2n} + O(1/n^2) \right) \,,
\end{equation}
where $\alpha \in \mathbb{C}$ and $n$ a large positive integer.


\providecommand{\href}[2]{#2}\begingroup\raggedright\endgroup

\end{document}